\documentclass[a4paper,11pt,pdf]{amsart}
\usepackage{enumerate, amsmath, amsfonts, amssymb, amsthm, thmtools, wasysym, graphics, graphicx, xcolor, frcursive,comment,bbm}

\usepackage{etex}

\definecolor{darkblue}{rgb}{0.0,0,0.7} 
\definecolor{darkred}{rgb}{0.7,0,0} 
\usepackage{hyperref}
\usepackage[all]{xy}
\usepackage[T1]{fontenc}

\usepackage{vmargin}            
\setmarginsrb{3.7cm}{3.2cm}{3.7cm}{3.1cm}{0cm}{0.6cm}{0cm}{0cm}

\usepackage{caption,lipsum}
\usepackage{graphicx}                  
\usepackage{pstricks,pst-plot,pst-text,pst-tree,pst-eps,pst-fill,pst-node,pst-math}
\usepackage{setspace}

\newcommand{\darkred}{\color{darkred}} 
\newcommand{\defn}[1]{\emph{\darkred #1}} 

\def\lt{\ell_{T}}

\def\leqt{{\leq_{T}}}
\def\Std{{\sf{Std}}}
\def\NC{{\sf{NC}}}
\def\Sort{{\sf{Sort}}}

\def\Cov{{\sf{Cov}}}

\def\NC{{\sf{NC}}}

\def\Red{{\mathrm{Red}}}
\def\Pol{{\sf{Pol}}}

\def\Cox{{\sf{Cox}}}

\def\Read{{\sf{Read}}}
\def\Sn{{\mathfrak{S}_{n+1}}}
\usepackage{etex}

\numberwithin{equation}{section}
\newtheorem{theorem}[equation]{Theorem}
\newtheorem{prop}[equation]{Proposition}
\newtheorem{lem}[equation]{Lemma}
\newtheorem{cor}[equation]{Corollary}

\theoremstyle{definition}
\newtheorem{definition}[equation]{Definition}
\newtheorem{rmq}[equation]{Remark}
\newtheorem{exple}[equation]{Example}

\newtheorem{question}[equation]{Question}

\title[Dual Garside structures and Coxeter sortable elements]{Dual Garside structures and Coxeter sortable elements}

\author{Thomas Gobet}
\address{Thomas Gobet, School of Mathematics and Statistics F07, University of Sydney NSW 2006, Australia.}
\email{thomas.gobet@sydney.edu.au}

\begin{document}
\maketitle

\begin{abstract}
In Artin-Tits groups attached to Coxeter groups of spherical type, we give a combinatorial formula to express the simple elements of the dual braid monoids in the classical Artin generators. Every simple dual braid is obtained by lifting an $S$-reduced expression of its image in the Coxeter group, in a way which involves Reading's $c$-sortable elements. It has as an immediate consequence that simple dual braids are Mikado braids (the known proofs of this result either require topological realizations of the Artin groups or categorification techniques), and hence that their images in the Iwahori-Hecke algebras have positivity properties. In the classical types, this requires to give an explicit description of the inverse of Reading's bijection from $c$-sortable elements to noncrossing partitions of a Coxeter element $c$, which might be of independent interest. The bijections are described in terms of the noncrossing partition models in these types. While the proof of the formula is case-by-case, it is entirely combinatorial and we develop an approach which reduces a uniform proof to uniformly proving a lemma about inversion sets of $c$-sortable elements.   
\end{abstract}
\tableofcontents

\section{Introduction}

Dual braid monoids provide a Garside structure on braid groups of finite reflection groups, and are therefore central for instance in the study of the word problem in generalized braid groups of complex reflection groups. They were introduced by Birman, Ko and Lee~\cite{BKL} in type $A_n$, Bessis, Digne and Michel~\cite{BDM} in type $B_n$, and Bessis~\cite{Dual} for the remaining finite real reflection groups as well as for complex reflection groups~\cite{Dualcplx}. They give rise to a so-called \textit{dual} approach for the study of both reflection groups and their braid groups. In this paper, we are interested in the dual braid monoids attached to finite real reflection groups.  

The dual approach consists in the study of a real reflection group viewed as a group generated by the set $T$ of all its reflections, instead of just the set $S$ of reflections corresponding to the walls of a chamber. It gives rise to a rich combinatorics involving the noncrossing partition lattices. More precisely, every Garside monoid comes equipped with a finite set of \textit{simple} elements, which may be defined as the (left) divisors of the Garside element, and form a lattice under left-divisibility in the monoid (for basics on Garside theory we refer the reader to \cite{Garside}). In the classical setting, i.e., when the generating set of the reflection group $W$ is a simple system $S$, this lattice is isomorphic to the lattice given by ordering $W$ by the (left) weak Bruhat order. In the dual setting, a choice of standard Coxeter element $c$ in $W$ is required, and the set of simples is in bijection with a subset $\NC(W,c)$ of $W$ associated with that choice of standard Coxeter element; the lattice structure on $\NC(W,c)$ coming from the lattice of simples is isomorphic to the lattice of noncrossing partitions of a Coxeter element $c$~\cite{Dual}. This gives rise to combinatorial models for the study of the simples and, more generally, of the braid groups defined by their dual presentation. We call the simples of the dual Garside structure the~\textit{simple dual braids}. 

Denote by $B(W)$ the braid group associated with the finite Coxeter group $(W,S)$. The dual braid monoid $B_c^*$ associated with a choice $c$ of Coxeter element embeds into its group of (left) fractions, which is isomorphic to $B(W)$. One of the difficulties of the dual approach is that it is hard in general to express the generators of $B_c^*$ (in bijection with $T$) or the simples (in bijection with $\NC(W,c)$) inside $B(W)$ in terms of the classical Artin generators. Indeed, generators of $B_c^*$ corresponding to simple reflections are sent to the corresponding Artin generators, but for other reflections there is no immediate description. In \cite[Proposition 3.13]{DG}, Digne and the author gave a formula to express the dual generators in the classical ones, but it does not give a braid word of smallest possible length in general. The aim of this paper is to give a combinatorial formula to express all the simples of $B_c^*$, not only the reflections, in the classical generators of $B(W)$ (see Theorem~\ref{thm:main} below). The formula for the simple $x_c$ associated to $x\in \NC(W,c)$ is given by lifting any $S$-reduced expression of $x$, where each generator $s$ in the reduced expression is replaced either by the corresponding Artin generator of $B(W)$ or its inverse, following a rule involving $c$-sortable elements. As an immediate consequence, the length of $x_c$ in $B(W)$ is equal to the length of $x$ in $W$ with respect to $S$, in particular the braid word which we obtain for $x_c$ is of smallest possible length (in fact, this shows that the length of $x_c$ in $B(W)$ is equal to the length of $x$ in $W$). 

The formula we give here is uniform and involves Reading's $c$-sortable elements (see~\cite{Reading}). To the best of our knowledge, this is the first application of these objects to the study of braid groups. There are several motivations for giving a formula to express the simple dual braids in the classical generating set: as we explain below, the formula has as an immediate consequence that simple dual braids are \textit{Mikado braids}. This result was conjectured in a joint work with Digne~\cite{DG}, where it was proven for every irreducible $W$ except in type $D_n$. Type $D_n$ was proven in a joint work with Baumeister~\cite{BG} and independently, Licata and Queffelec~\cite{LQ} proved it in types $A_n$, $D_n$ and $E_n$. The proofs there use either topological realizations of the braid groups (i.e., in terms of Artin braids) or categorification techniques. There are indeed indications that Mikado braids play an important role in the categorifications of the braid groups (see~\cite{Twisted}, \cite{LQ}). Another motivation consists of searching for a construction of the dual braid monoids which would avoid using the classification. In this perspective, a uniform proof of Lemma~\ref{lem:casparcasbis} of this paper would be of interest since the formula which we give is case-free, but its proof relies on Lemma~\ref{lem:casparcasbis} whose proof uses the classification. Finally, there are computational motivations: the formula which we give allows one to compute a simple expressed in the classical Artin generators much quicker than with the known approaches (see Section~\ref{dual} below).  

Reading defined and studied $c$-sortable elements in the framework of what is nowadays called the \textit{Coxeter-Catalan combinatorics}. Coxeter sortable elements form a combinatorially defined subset of $W$ (attached to a choice of standard Coxeter element $c$), which is in bijection with the set $\NC(W,c)$ of noncrossing partitions of $c$ and with $c$-clusters; they were indeed introduced to construct a bijection between the set $\NC(W,c)$ and the set of $c$-clusters. Their definition allows a natural and powerful way of arguing by induction, sometimes called \textit{Cambrian recurrence} (see Lemmatas~\ref{lem:read1} and \ref{lem:read2} below). Such a procedure allows one to prove results concerning noncrossing partitions for all standard Coxeter elements simultaneously, which is difficult to perform if we stay in the world of noncrossing partitions. This strategy is used to prove the main result (Theorem~\ref{thm:main} below). Let $\Read$ denote Reading's bijection between $c$-sortable elements and $c$-noncrossing partitions. 

Mikado braids were defined (in spherical type) in \cite{DG} but they already appeared before in the work of Dyer~\cite{Dyernil} and Dehornoy~\cite{D}. In spherical type, a Mikado braid can be defined as a quotient of two simple elements of the classical Garside structure, that is, two positive lifts of the elements of $W$ in $B(W)$. The optimal definition is the one appearing in work of Dyer~\cite[Section 9.1]{Dyernil}, which associates a Mikado braid to two elements $x,y$ of $W$ (at least if $W$ is finite): a Mikado braid $x_{N(y)}$ is obtained by lifting any $S$-reduced expression $s_1 s_2 \cdots s_k$ of $x$ to the element $$\mathbf{s}_1^{\varepsilon_1} \mathbf{s}_2^{\varepsilon_2}\cdots \mathbf{s}_k^{\varepsilon_k}\in B(W),$$ where $\mathbf{s}_i$ is the classical Artin generator of $B(W)$ associated with $s_i$ and $\varepsilon_i\in\{ \pm 1 \}$; the rule to decide whether $\varepsilon_i=1$ or $-1$ involves the inversion set $N(y)$ of the element $y$ (see Section~\ref{mikado} for a precise definition). The element $y$ is not unique in general, that is, we may have $x_{N(y)}=x_{N(y')}$ for $y\neq y'\in W$. For infinite groups one replaces inversion sets $N(y)$ by biclosed sets of positive roots (which turn out to be exactly inversion sets in the case where $W$ is finite). Also note that by definition, Mikado braids have the same length in the braid groups (with respect to the generating set given by the Artin generators and their inverses) as their images in the Coxeter group (with respect to $S$). 

As explained above, simple dual braids were conjectured and then proven to be Mikado braids, but in an indirect way: if $x\in \NC(W,c)$ and $x_c\in B_c^*\subseteq B(W)$ denotes the corresponding simple dual braid, then we have $x_c=x_{N(y)}$ for some $y\in W$ there is no hint of what "the" (or "a such") $y$ could be (but rather algorithms to find a possible $y$, which rely on the topological models for the Artin groups in the classical types; see for instance~\cite[Proposition 5.7]{DG}). Our formula provides an answer to that question: 

\newtheorem*{thm:main}{Theorem \ref{thm:main}}
\begin{thm:main}[Expressing simple dual braids in the classical generators]
Let $c$ be a standard Coxeter element. Let $x\in \NC(W,c)$, let $y:=x^{-1} c$ (which also lies in $\NC(W,c)$). Then $$x_c= x_{\Read^{-1}(y)}.$$
\end{thm:main}    

The bijection $\Read$ is defined uniformly, hence the statement of the Theorem does not require the classification of finite Coxeter groups. Unfortunately, it seems that there is no explicit uniform description of its inverse $\Read^{-1}$ in the literature. In Section~\ref{explicit} below, we provide a combinatorial description of $\Read^{-1}$ in the classical types in terms of noncrossing partitions. These descriptions, which might be of independent interest, are required to prove the following combinatorial Lemma about $c$-sortable elements. The proof of it is the only one where we need to use the classification:

\newtheorem*{lem:casparcasbis}{Lemma \ref{lem:casparcasbis}}
\begin{lem:casparcasbis}
Let $c$ be a standard Coxeter element. Let $x\in \NC(W,c)$, $y:=x^{-1}c$. Let $s$ be initial in $c$, that is, $\ell(sc) < \ell(c)$. Then $$x^{-1} sx\text{ lies in }N(y)\text{ if and only if it lies in }N(\Read^{-1}(y)).$$ 
\end{lem:casparcasbis}

It would be interesting to extend these results to the affine Coxeter groups of type $\widetilde{A}_n$ and $\widetilde{C}_n$ for the choices of Coxeter element for which there is a dual braid monoid~(see~\cite{Dig},\cite{Dig1}). In these cases it is not known whether simple dual braids are Mikado braids.

Mikado braids have interesting homological properties in the categorical incarnations of the braid groups (see~\cite{Twisted}, \cite{LQ}). For instance, the fact that they lie in the heart of the canonical $t$-structure on the bounded homotopy category of Soergel bimodules can be used to derive positivity properties of their images in the Iwahori-Hecke algebra of the Coxeter system (see Section~\ref{mik} below for more details). Combining this fact with the fact that simple dual braids are Mikado braids, we obtain that the images of simple dual braids in the Hecke algebra have positivity properties. In type $A_n$, these images yield a basis of the Temperley-Lieb quotient of the Hecke algebra (see~\cite{Gob_jktr}).  
 \\
~\\
{\bf{Acknowledgements}}. The author thanks Anthony Henderson, Christophe Hohlweg, Ivan Marin, Nathan Reading, Christian Stump and Nathan Williams for useful discussions. He thanks Fran\c{c}ois Digne and Jean Michel for providing computer programs to check Lemma~\ref{lem:casparcasbis} in the exceptional types. This paper was written in part when the author was funded by the ANR Geolie (project ANR-15-CE40-0012) at the Universit\'e de Lorraine in Nancy, and finalized while being supported by an Australian Research Council (ARC) grant (grant number DP170101579) at the University of Sydney. He thanks both organizations for their financial support.  

\section{Coxeter groups and Coxeter elements}

\subsection{Coxeter groups and Artin groups}

A \defn{Coxeter system} $(W,S)$ is a group $W$ generated by a finite set $S$ of elements satisfying $s^2=e$ for every $s\in S$, subject to additional \defn{braid relations}: given $s,t\in S$, $s\neq t$, the braid relation of length $m_{s,t}=m_{t,s}\in\{2,3,\dots\}\cup\{\infty\}$ between $s$ and $t$ is the relation $st\cdots = ts\cdots$, where $st\cdots$ (resp; $ts\cdots$) is the strictly alternating product of $s$ and $t$ with $m_{s,t}$ factors and beginning by $s$ (resp. by $t$). If $m_{s,t}=\infty$ then no relation between $s$ and $t$ is imposed. We denote by $\ell:W\longrightarrow\mathbb{Z}_{\geq 0}$ the length function with respect to the generating set $S$ of $W$ and by $\leq$ the Bruhat order on $W$. The set $T=\bigcup_{w\in W} wSw^{-1}$ is the set of \defn{reflections} in $W$. A Coxeter group $(W,S)$ is \defn{reducible} if there is a nontrivial partition $S= S_1 \overset{\cdot}{\cup} S_2$ such that $W= \langle S_1\rangle \times \langle S_2 \rangle$ (it is a general fact that every subgroup $W_I$ generated by a subset $I\subseteq S$, or even every subgroup $W'\subseteq W$ generated by a set of reflections, is again a Coxeter group; see~\cite{Dyer_sub}). Otherwise it is \defn{irreducible}. Every Coxeter group is isomorphic to a finite direct product of irreducible ones, in an essentially unique way. For more on Coxeter groups we refer the reader to~\cite{Bou} or~\cite{Humph}.   

Finite irreducible Coxeter groups are classified in four infinite families ($A_n$, $n\geq 3$), ($B_n$, $n\geq 3$), ($D_n$, $n\geq 4$), ($I_2(m)$, $m\geq 5$) and six exceptional groups $E_6$, $E_7$, $E_8$, $F_4$, $H_3$, $H_4$. 

Given a (not necessarily finite) Coxeter system $(W,S)$, on defines the \defn{Artin-Tits group} $B(W)$ attached to $(W,S)$ as the group generated by a copy $\mathbf{S}=\{\mathbf{s}~|~s\in S\}$ of $S$, whose elements are only subject to the braid relations of $W$. That is, the only relations are $$\underbrace{\mathbf{s} \mathbf{t}\cdots}_{m_{s,t}~\text{factors}}=\underbrace{\mathbf{t} \mathbf{s}\cdots}_{m_{s,t}~\text{factors}},$$ for all $s,t\in S$. As this presentation is positive, let $B(W)^+$ denote the monoid generated by the same presentation as $B(W)$. It is a general fact that $B(W)^+\subseteq B(W)$ (see~\cite{Paris}, \cite{Jensen}). If $W$ is finite, the group $B(W)$ (or $W$) is said to be of spherical type. 

When $W$ is finite, $B(W)^+$ is a \defn{Garside monoid} (see~\cite{Garside}), implying that $B(W)$ can be realized as the group of (left) fractions of $B(W)^+$. Garside structures are central in the study of the word problem in braid groups. They provide normal forms for the elements of $B(W)^+$, which can be used to define (and compute) normal forms for the elements of $B(W)$. Garside structures are not unique in general; in Section~\ref{dual} below, we recall the definition of the dual braid monoids, which provide alternative Garside structures on $B(W)$. We do not introduce all the machinery of Garside monoids here since our purposes can be achieved without it, but rather refer the reader to~\cite{Garside} for more on the topic.  

For $x\in W$, define the \defn{positive lift} $\mathbf{x}\in B(W)$ of $x$ as follows: choose a reduced expression $s_1 s_2\cdots s_k$ of $x$ and set $\mathbf{x}:=\mathbf{s}_1 \mathbf{s}_2\cdots \mathbf{s}_k$. Since any two reduced expressions of $x$ can be related by a sequence of braid relations, the element $\mathbf{x}$ is well-defined. If $W$ is finite, in the Garside-theoretic language, the set $\{\mathbf{x}~|~x\in W\}$ is the set of \defn{simple elements} or \defn{simples} of the classical braid monoid $B(W)^+$, that is, the set of divisors of the Garside element $\Delta$ (which is equal to $\mathbf{w_0}$, where $w_0$ is the longest element of $W$). 

\subsection{Coxeter elements and absolute order}

In this section, we recall the definition and properties of Coxeter elements, absolute order and reflection subgroups. The first definition below is borrowed from~\cite{BDSW}:

\begin{definition}[Coxeter elements]
Let $(W,S)$ be a Coxeter system. Let $T=\bigcup_{w\in W} w S w^{-1}$ denote the set of reflections in $W$.
\begin{enumerate}
\item An element $c\in W$ is a \defn{Coxeter element} if there exists $$S'=\{s_1, \dots, s_n\}\subseteq T$$ such that $c= s_1 s_2\cdots s_n$ and $S'$ is a simple system for $W$ (i.e., $(W,S')$ is a Coxeter system). We denote by $\Cox(W,S)$ the set of Coxeter elements in $W$.
\item An element $c\in W$ is a \defn{standard Coxeter element} if it is the product of all the elements of $S$ in some order. We denote by $\Std(W,S)$ the set of standard Coxeter elements in $W$. It is clear that $\Std(W,S)\subseteq\Cox(W,S)$. 
\end{enumerate}
\end{definition}

\begin{rmq}
The first definition above might look independent of $S$, but it does not just depend on $W$ as abstract group: one requires to have fixed the set of reflections $T$ for $W$, and since the set $T$ cannot be obtained just from the abstract group $W$ (that is, there are Coxeter groups $W, W'$ which are isomorphic as abstract groups, but not as Coxeter groups; see for instance~\cite[Section 2.3]{BGRW} and the references therein), the easiest way to do it is to fix a simple system $S$ for $W$ and define $T$ as the set of conjugates of the elements of $S$ as we did above. Note that if $S'\subseteq T$ is a simple system for $W$, then $\Cox(W, S)=\Cox(W, S')$ (as an immediate consequence of~\cite[Lemma 3.7]{Mul}).
\end{rmq}

Every Coxeter element $c=s_1 s_2\cdots s_n$ has reflection length equal to $n=|S|$; this can be seen for instance using the main result of~\cite{Dyer_refn}.

\begin{definition}[Absolute order]
The \defn{absolute order} $\leq_T$ on $W$ is defined by $u\leq_T v$ if and only if $$\ell_T(u)+\ell_T(u^{-1}v)=\ell_T(v),$$ where $\ell_T: W\longrightarrow \mathbb{Z}_{\geq 0}$ is the length function with respect to the generating set $T$ of $W$ and $u,v\in W$. We denote by $\Red_T(u)$ the set of $T$-reduced expressions of $u$. 
\end{definition}

\begin{definition}[Parabolic Coxeter elements]\label{def:par}
Let $(W,S)$ be a Coxeter system. Let $T=\bigcup_{w\in W} w S w^{-1}$ be the set of reflections of $W$.
\begin{enumerate}
\item An element $w\in W$ is a \defn{parabolic Coxeter element} if it exists a simple system $S'=\{ s_1, s_2,\dots, s_n\}\subseteq T$ for $W$ and $m\leq n$ such that $w= s_1 s_2\cdots s_m$. 
\item An element $w\in W$ is a \defn{prefix} of a Coxeter element $c\in\Cox(W,S)$ if $w\leq_T c$. 
\end{enumerate}
\end{definition}

\begin{rmq}
A parabolic Coxeter element is always a prefix of a Coxeter element since in the notations of $(1)$ above, we have $s_1 s_2\cdots s_m \leqt s_1 s_2\cdots s_n$. However the converse is false in general (see for instance~\cite[Example 5.7]{HK}), but it holds if $W$ is finite (see~\cite[Corollary 3.6]{DG}).  
\end{rmq}

\subsection{Parabolic and reflection subgroups}

\begin{definition}[Parabolic subgroup]
A reflection subgroup $W'\subseteq W$ is \defn{parabolic} if $W'=\langle s_1, s_2, \dots, s_m \rangle$, where $s_1, \dots, s_m$ is as in Definition~\ref{def:par} (1). That is, if there is a set of reflections generating $W'$ which can be extended to a simple system $S'$ for $W$. There is a more classical definition as follows: let $I\subseteq S$. Then $W_I:=\langle s~|~ s\in I\rangle$ is a \defn{standard parabolic subgroup}, and any conjugate of such a group is a parabolic subgroup. The equivalence of the two definitions is false in general (see~\cite[Example 2.2]{Gob_cycle} for a counterexample), but it holds for finite groups (and for a large family of infinite Coxeter groups called $2$-spherical; see~\cite[Section 4]{BGRW}). 
\end{definition}

As mentioned above, every reflection subgroup $W'\subseteq W$ of a Coxeter system $(W,S)$ is again a Coxeter group (see~\cite{Dyer_sub}). It comes equipped with a canonical set of generators as a Coxeter group, which is defined as follows. Given $y\in W$, define $$N(y):=\{t\in T~|~\ell(ty) < \ell(y)\}.$$
The elements of $N(y)$ are the \defn{(left) inversions} of $y\in W$ and we have $|N(y)|=\ell(y)$. Given a reflection subgroup $W'\subseteq W$, its canonical set of Coxeter generators is given by (see~\cite[Theorem 3.3]{Dyer_sub}) $$S':=\{ t\in W'\cap T~|~N(t)\cap W'=\{t\}\}.$$
Morevoer, the reflections of $W'$ (viewed as a Coxeter group) coincide with $W'\cap T$. 

Assume that $(W,S)$ is finite. Let $V$ be the geometric representation of $(W,S)$. Recall that it is a faithful representation $V$ of $W$, preserving a symmetric bilinear form $(\cdot, \cdot)$, in such a way that the elements of $T$ act on $V$ by geometric reflections (see~\cite[Section 5.3]{Humph}). To every $t\in T$ corresponds a \defn{root} $\alpha_t\in V$. Given $x\in W$ we denote by $V^x$ the set of fixed points of $V$ by $x$. There is the following well-known characterization of parabolic subgroups:

\begin{prop}\label{prop_parab}
Let $(W,S)$ be a finite Coxeter system. A subgroup $W'\subseteq W$ is parabolic if and only if there is a subset $E\subseteq V$ such that $$W'=C_W(E):=\{ w\in W~|~w(v)=v~\forall v\in E\}.$$
\end{prop}

The subsets $E$ can be chosen to be intersections of reflecting hyperplanes. 

\begin{rmq}\label{rmq_centralizer} 
Note that in particular, if $W$ is finite, then $P(x)=C_W(V^x)$ for all $x\in W$. 
\end{rmq}

\begin{definition}[Parabolic closure]
Let $(W,S)$ be a finite Coxeter system. Let $x\in W$. The \defn{parabolic closure} $P(x)$ of $x$ is the smallest parabolic subgroup of $W$ containing $x$. This is well-defined since parabolic subgroups of finite Coxeter groups are known to be stable by intersection (for infinite Coxeter groups, this is true with the classical definition, but open with the definition of parabolic subgroups which we chose to use in this paper). 
\end{definition}

\subsection{Noncrossing partitions}

From now on, we assume $(W,S)$ to be finite. 

\begin{definition}[Noncrossing partitions]
For $c\in\Cox(W,S)$, we set $$\NC(W,c):=\{x\in W~|~x\leq_T c\}$$ and call the elements of this set the \defn{$c$-noncrossing partitions}.
\end{definition}

The reason for calling these elements \textit{noncrossing partitions} comes from the fact that in type $A_n$, the poset $(\NC(W,c),\leq_T)$ is isomorphic to the lattice of noncrossing partitions of the cycle $c$ ordered by refinement (see~\cite{Biane} or Section~\ref{graph_a} below). In general we have: 

\begin{theorem}[Bessis~\cite{Dual}, Brady and Watt~\cite{BW}, Reading~\cite{Reading2}]
The partially ordered set $(\NC(W, c), \leq_T)$ is a lattice. 
\end{theorem} 

If $W$ is finite, then $(\NC(W,c),\leq_T)\cong (\NC(W,c'),\leq_T)$ for all $c,c'\in\Cox(W,S)$ (if $c$ and $c'$ are conjugate this is clear, but with the definition of Coxeter elements which we adopted for this paper two Coxeter elements might not be conjugate. One nevertheless still has the isomorphism of posets of noncrossing partitions, see~\cite{RRS}). 

The following result is due to Carter~\cite{Carter}:

\begin{theorem}[Carter~\cite{Carter}]\label{carter}
Let $u,v\in W$. Let $u=t_1 t_2\cdots t_k$ be an expression of $u$ in the generating set $T$ of $W$.  
\begin{enumerate}
\item The expression $t_1 t_2 \cdots t_k$ is $T$-reduced (i.e., $k=\ell_T(u)$) if and only if the roots $\alpha_{t_1}, \alpha_{t_2}, \dots, \alpha_{t_k}\in V$ are linearly independent. In this case, the reflection subgroup $\langle t_1, t_2, \dots, t_k \rangle$ of $W$ has rank $k$, 
\item We have $\ell_T(u)=\dim V - \dim V^u=\mathrm{codim}(V^u)$,
\item We have $u\leq_T v \Rightarrow V^v\subseteq V^u$,
\item If $t\in T$, then $V^v\subseteq V^t\Leftrightarrow t\leq_T v$.
\end{enumerate}
\end{theorem} 
Note that $V^t$ is the reflecting hyperplane $H_t$ of $t$. In general, the inclusion $V^v\subseteq V^u$ for $u,v\in T$ does not imply that $u\leq_T v$: for instance, Coxeter elements always have a trivial set of fixpoints, but there are element failing to be below a given Coxeter element for the absolute order in general. The first statement in the above Theorem is sometimes called ``Carter's Lemma''.  

We recall some well-known facts about parabolic subgroups associated to noncrossing partitions. Parabolic subgroups can also be described as centralizers of subspaces in the geometric representation $V$ as recalled in Proposition~\ref{prop_parab} above. One can associate to every $x\in\NC(W,c)$ its parabolic closure $P(x)$. Below, we show that it can also be defined as $$P(x)=\langle t\in T~|~t\leq_T x\rangle$$ and that given $x, y\in \NC(W,c)$, we have $x\leq_T y$ if and only if $P(x)\subseteq P(y)$. Indeed, let $w\in \langle t\in T~|~t\leq_T x\rangle$. By point $(3)$ of the Theorem above, we have that $V^x\subseteq V^w$, implying that $w\in C_W(V^x)$. Conversely, let $w\in C_W(V^x)$. Let $t_1 t_2\cdots t_k$ be a $T$-reduced decomposition of $w$. We have $t_i\leq_T w$ for all $i$, implying by point (3) of the above Theorem again that $V^x\subseteq V^w\subseteq V^{t_i}$ for all $i$. Using point (4) we get that $t_i\leq_T x$ for all $i$, hence that $t_i\in \langle t\in T~|~t\leq_T x\rangle$ for all $i$, which implies that $w$ itself lies in $\langle t\in T~|~t\leq_T x\rangle$. This shows that $\langle t\in T~|~t\leq_T x\rangle=C_W(V^x)=P(x)$. Now assume that $x\leq_T y$. Then by point (2) above we have $V^y\subseteq V^x$, hence $P(x)=C_W(V^x)\subseteq C_W(V^y)=P(y)$. The converse follows by the main result of~\cite{BW2}. 

\begin{rmq}\label{t_dans_pc}
By Theorem~\ref{carter}~(2), each Coxeter element has a trivial set of fixed points in $V$. By point (4) of the same Theorem it implies that $T\subseteq \NC(W,c)$ for all $c\in\Cox(W,S)$.
\end{rmq}

\section{Dual braid monoids}\label{dual}

Dual braid monoids provide alternative Garside structures on Artin-Tits groups of spherical type. More precisely, a dual braid monoid attached to a finite Coxeter group embeds into the corresponding Artin-Tits group and is isomorphic to its group of (left) fraction. It was initially introduced by Birman, Ko and Lee~\cite{BKL} in type $A_n$, by Bessis, Digne and Michel~\cite{BDM} in type $B_n$ and by Bessis~\cite{Dual} for all finite Coxeter groups. The idea is to take as generating set for $W$ (resp. for the dual braid monoid) the set $T$ (resp. a copy of the set $T$) of reflections. A choice of standard Coxeter element $c\in\Std(W,S)$ is required to define the dual braid monoid. We recall that all Coxeter systems are assumed to be finite. 

We recall from Remark~\ref{t_dans_pc} that $T\subseteq \NC(W,c)$ for all $c\in\Cox(W,S)$.

\begin{definition}[Dual braid monoid]
Let $c\in\Std(W,S)$. The \defn{dual braid monoid} $B_c^*$ associated to $W$ and $c$ is defined by $$B_c^*:=\langle t_c\in T_c~ |~t_c t'_c= (tt't)_c t_c~\text{if }tt'\in\NC(W,c)\rangle.$$ 
\end{definition}

We recall some properties of dual braid monoids; the reader is referred to \cite{Dual} or \cite[Section 3]{DG} for more details and proofs. The monoid $B_c^*$ is a Garside monoid with group of (left) fractions isomorphic to $B(W)$, in particular it embeds into $B(W)$. The embedding $\iota: B_c^*\hookrightarrow B(W)$ is hard to describe combinatorially; for each $s\in S$ we have $\iota(s_c)=\mathbf{s}$, but in the case where $t\in T\backslash{S}$, there is no immediate way to express $\iota(t_c)$ in the generators $\mathbf{S}\cup \mathbf{S}^{-1}$ of $B(W)$. One way it achieve this purpose is to use the dual braid relations inductively. For instance, if $W$ has type $A_2$ with Coxeter element $c=s_1 s_2$, then we have $(s_1)_c (s_2)_c=(s_1 s_2 s_1)_c (s_1)_c$, hence $\iota( (s_1 s_2 s_1)_c)=\mathbf{s}_1 \mathbf{s}_2 \mathbf{s}_1^{-1}$. Another way is given by Proposition~\ref{prop_atomes} below. 

While noncrossing partitions can be defined for an arbitrary Coxeter element, for dual braid monoids the Coxeter element is required to be standard (see~\cite[Remark 5.11]{DG}). The set of \defn{simple elements} or \defn{simples} $\mathrm{Div}(c)$ of $B_c^*$ is in one-to-one correspondence with $\NC(W,c)$. If is defined as the set of elements in $B_c^*$ which (left-)divide the lift $c_c= (s_1)_c (s_2)_c\cdots (s_n)_c$ of the Coxeter element $c=s_1 s_2\cdots s_n$. The element $x_c$ corresponding to $x$ is built as follows: given a $T$-reduced expression $t_1 t_2\cdots t_k$ of $x$, the element $(t_1)_c (t_2)_c \cdots (t_k)_c$ turns out to be independent of the chosen $T$-reduced expression, and we denote it by $x_c$ (this follows from the transitivity of the Hurwitz action on the reduced decompositions of parabolic Coxeter elements, see~\cite{Dual}). We abuse notation and simply write $x_c$ instead of $\iota(x_c)$. We want to find an efficient way of expressing $x_c$ in the classical Artin generators of $B(W)$. For generators of $B_c^*$, it can be done in the following way~\cite[Proposition 3.13]{DG}:

\begin{prop}\label{prop_atomes}
Let $c=s_1 s_2 \cdots s_n\in\Std(W,S)$. We have $$T_c=\{ \mathbf{s}_1 \mathbf{s}_2 \cdots \mathbf{s}_i \mathbf{s}_{i+1} \mathbf{s}_i^{-1} \cdots \mathbf{s}_2^{-1} \mathbf{s}_1^{-1}~|~0\leq i\leq 2|T|\},$$ where the indices are taken modulo $n$. 
\end{prop}

However this does not give braid words of smallest possible length in general. 

We prove a few Lemmatas which will be useful later on. Most of them are well-known, but we provide proofs for completeness. 

\begin{lem}\label{dual1}
Let $I\subseteq S$. Let $c\in \Std(W,S)$ and let $c_I$ denote the restriction of $c$ to the standard parabolic subgroup $W_I:=\langle s~|~s\in I\rangle$ (obtained by deleting the letters in a reduced expression of $c$ which are not in $I$). Note that we have $c_I\leq_T c$. 
\begin{enumerate}
\item There is an embedding $B_{c_I}^*\subseteq B_c^*$ induced by $t_{c_I} \mapsto t_c$, $t\in T\cap W_I$. 
\item Let $x\in \NC(W_I,{c_I})\subseteq \NC(W,c)$. Then $x_c=x_{c_I}$. 
\end{enumerate}
\end{lem}
\begin{proof}
Since dual braid relations in $B_{c_I}^*$ are dual braid relations in $B_c^*$ (because $c_I\leq_T c$), there is a map $\varphi_I: B_{c_I}^*\longrightarrow B_c^*$.

The reflection length of $x$ in $W_I$ is equal to its reflection length in $W$ and the reflections in $W_I$ are given by $T_I:=T\cap W_I$. By Theorem~\ref{carter}, we have that every $T$-reduced decomposition of $x$ has all its reflections in $W_I$. To show the second part, it therefore suffices to show that for $t\in T_I$, we have $t_{c_I}=t_c$. This is true for simple reflections, implying that $$(c_I)_{c_I}= (c_I)_c.$$
We have $t\leq_T c_I$ by Remark~\ref{t_dans_pc}, hence the result follows by the above equality and the transitivity of the Hurwitz action on $\Red_{T_I} (c_I)$ (and its lift in the dual braid monoid; see~\cite[Theorem 3.10, Corollary 3.12]{DG}): $t_{c_I}$ can be obtained by a successive application of dual braid relations to the reduced decomposition $(s_1)_c (s_2)_c \cdots (s_k)_c$ of $(c_I)_{c_I}=({c_I})_c$, which can be performed either in $B_{c_I}^*$ or in $B_c^*$. This shows the second part. It also shows that the composition
 $$B_{c_I}^*\hookrightarrow B(W_I)\hookrightarrow B(W)$$
is equal to the composition of $\varphi_I$ with the embedding $B_c^*\hookrightarrow B(W)$, proving the first claim. 

\end{proof}

\begin{lem}\label{dual2}
Let $c\in\Std(W,S)$ and let $s\in S$ be initial in $c$ (that is, $c$ has an $S$-reduced expression beginning by $s$). Let $x\in \NC(W,c)$. Note that $sxs\in \NC(W,scs)$. Then $$(sxs)_{scs}=\mathbf{s}^{-1} x_c \mathbf{s}.$$
\end{lem}

\begin{proof}
Again, it suffices to show the result in the case where $x\in T$. If $x\in S$, then the claim is true: if $x=s$ it is clear, while if $x=s'\neq s$, we have that $s$ appears after $s'$ in $scs$ since $s$ is final in $scs$, hence that $s's\leq_T scs$, implying that we have a dual braid relation $$(s')_{scs} s_{scs}= s_{scs} (ss's)_{scs},$$ and hence $$(sxs)_{scs}=(ss's)_{scs}=\mathbf{s}^{-1} \mathbf{s'} \mathbf{s}=\mathbf{s}^{-1} x_c \mathbf{s}.$$ The result for all $t\in T$ follows by Hurwitz transitivity. 
\end{proof}

\section{Coxeter sortable elements}

From now on, unless explicitly stated otherwise, we will always assume $(W,S)$ to be finite. We recall from~\cite{Reading} the definition and basic properties of $c$-sortable elements. Fix an $S$-reduced decomposition $s_1 s_2 \cdots s_n$ of a standard Coxeter element $c\in \Std(W,S)$. Consider the semi-infinite word $$c^\infty= s_1 s_2 \cdots s_n| s_1 s_2 \cdots s_n | s_1 s_2\cdots s_n |\cdots .$$ Let $w\in W$. The \defn{$c$-sorting word} for $w$ is the lexicographically first $S$-reduced expression of $w$ appearing as a subword of $c^\infty$. We write $w=w_1 w_2\cdots w_k$, where $w_i$ is the subword of the $c$-sorting word for $w$ coming from the $i$-th copy of $c$ in $c^\infty$ and $k$ is maximal such that this subword is nonempty. Note that each $w_i$ defines a subset of $S$, consisting of the letters in the word $w_i$. We abuse notation and also denote this set by $w_i$. 

\begin{definition}[Coxeter sortable elements]
We say that $w$ is \defn{$c$-sortable} if $w_k \subseteq w_{k-1} \subseteq \dots \subseteq w_2 \subseteq w_1$. 
\end{definition}

\begin{exple}
Let $W$ be of type $A_3$, $S=\{s_1, s_2, s_3\}$ with $s_1 s_3=s_3 s_1$. Let $c=s_3 s_1 s_2$. Then $w=s_3 s_1 s_2| s_3 s_1$ is $c$-sortable, with $w_1=s_3 s_1 s_2$, $w_2=s_3 s_1\subseteq w_1$. The element $s_2 s_1$ is not $c$-sortable since if it was, we would have $s_2 s_1=w_1$, but in $c$ the letter $s_1$ appears before $s_2$ (and they do not commute). 
\end{exple}

We denote by $\Sort_c(W)$ the set of $c$-sortable elements. It is independent of the reduced expression of the Coxeter element $c$ which was chosen (any two reduced expressions of $c$ are related by commutation of letters, and the same holds for the $c$-sorting words). We say that a simple reflection $s\in S$ is \defn{initial} in $c$ if $c$ has an $S$-reduced expression beginning by $s$. 

The following two Lemmatas, which are immediate consequences of the definition of $c$-sortable elements, are crucial in the theory of $c$-sortable elements. For $s\in S$, we denote by $W_{<s>}$ the maximal parabolic subgroup of $W$ generated by $S\backslash\{ s\}$. 

\begin{lem}\label{lem:read1}
Let $s$ be initial in $c$. Let $w\in W$ be such that $\ell(sw) > \ell(w)$. Then $w\in\Sort_c(W)\Leftrightarrow w\in \Sort_{sc}(W_{<s>})$.
\end{lem}

\begin{lem}\label{lem:read2}the s
Let $s$ be initial in $c$. Let $w\in W$ be such that $\ell(sw) < \ell(w)$. Then $w\in\Sort_c(W)\Leftrightarrow sw\in \Sort_{scs}(W)$.
\end{lem}

The two above Lemmatas allow one to argue by induction on the rank of $W$ and the length of $c$-sortable elements (without fixing a Coxeter element). 

Reading defined a bijection $\Read: \Sort_c(W)\longrightarrow \NC(W,c)$ as follows. Consider the $c$-sorting word $s_{i_1} \cdots s_{i_\ell}$ for $w$. Put a total order on $N(w)$ such that the $j$-th element is $s_{i_1} s_{i_2}\cdots s_{i_j} s_{i_{j-1}} \cdots s_{i_1}$. Set $D_R(w):=\{s\in S~|~\ell(ws) < \ell(w)\}$ and $$\Cov(w):=\{wsw^{-1}~|~s\in D_R(w)\}\subseteq N(w).$$ The elements of $\Cov(w)$ are the \defn{cover reflections} of $w$. Order the elements of $\Cov(w)$ with respect to the total order on $N(w)$ defined above, say $t_1, \dots, t_k$, and define an element $x\in W$ as the product $t_1 t_2\cdots t_k$.

\begin{theorem}[Reading~\cite{Reading}]\label{thm_reading}
With the above notations, the assignment $w\mapsto x:=\Read(w)$ defines a bijection $\Sort_c(W)\longrightarrow \NC(W,c)$. The word $t_1 t_2\cdots t_k$ is $T$-reduced and $\{t_1, t_2, \dots, t_k\}$ is the set of canonical generators of the parabolic subgroup $P(x)\subseteq W$. 
\end{theorem}

We denote by $S_c: \NC(W,c)\longrightarrow \Sort_c(W)$ the map which is the inverse of the bijection $\Read$. In Section~\ref{explicit}, we will give an explicit description of $S_c$ in the classical types, using the combinatorial models of noncrossing partitions.

The following Lemma is the part which requires a case-by-case investigation. The proof is postponed to Section~\ref{explicit}:

\begin{lem}\label{lem:casparcasbis}
Let $W$ be a finite Coxeter group, $c\in W$ a Coxeter element and $s\in S$ be initial in $c$. Let $x\in\NC(W, c)$, $y:=x^{-1}c$. Then $$x^{-1} s x\in N(y)~\Leftrightarrow~ x^{-1} sx\in N(S_c(y)).$$
\end{lem}

\begin{rmq}\label{rmq:refor}
Note that in the assumptions of the Lemma, we have $x^{-1}sx\in N(y)$ if and only if $s\notin N(x)$; indeed, since $s$ is initial in $c$ we have $$s\in N(c)=N(xy)=N(x) + x N(y) x^{-1},$$ where $+$ denotes symmetric difference. 
\end{rmq}



\section{Mikado braids and a closed formula for simple dual braids}\label{mikado}

The aim of this section is to prove the main result, namely Theorem~\ref{thm:main}, assuming Lemma~\ref{lem:casparcasbis} (which will be proven by a case-by-case analysis in Section~\ref{explicit}). We begin by recalling the definition and properties of Mikado braids. 

\subsection{Mikado braids}\label{mik}

Mikado braids appeared in work of Dehornoy~\cite{D} in type $A_n$ and in unpublished work of Dyer~\cite{Dyernil} for abritrary Coxeter groups. A detailed study was initiated in \cite{DG}, \cite{Twisted}. Mikado braids are elements of an Artin-Tits group attached to an arbitrary Coxeter group. In this paper, we recall that we restrict to the spherical type, where the definition is easier and has several equivalent formulations (for a general definition see~\cite[Definition 3.5]{Twisted})

\begin{lem}[Dyer~\cite{Dyernil}]
Let $W$ be a finite Coxeter group. Let $B(W)$ be the corresponding Artin-Tits group. Let $x,y\in W$. Define an element $x_{N(y)}$ of $B(W)$ as follows: choose a reduced expression $s_1 s_2\cdots s_k$ of $x$ and set $$x_{N(y)}= \mathbf{s}_1^{\varepsilon_1} \mathbf{s}_2^{\varepsilon_2}\cdots \mathbf{s}_k^{\varepsilon_k},$$ where $\varepsilon_i= -1$ if $s_k s_{k-1}\cdots s_i s_{i+1} \cdots s_{k-1} s_k\in N(y)$ and $\varepsilon_i=1$ otherwise. Then $x_{N(y)}$ is well-defined. 
\end{lem}

For arbitrary $W$, the set $N(y)$ in the above definition is replaced by any biclosed set of roots (these sets are precisely inversion sets when the group is finite. In general there are many more biclosed sets than inversion sets). 

\begin{definition}
Let $\beta\in B(W)$. We say that $\beta$ is a \defn{Mikado braid} if there are $x,y\in W$ such that $\beta=x_{N(y)}$. 
\end{definition}

\begin{exple}\label{ex_mik}
Let $W$ be of type $A_3$, with set $S$ of generators given by $s_1, s_2, s_3$ (where $s_1 s_3=s_3 s_1$). We claim that $\beta:=\mathbf{s}_2 \mathbf{s}_1 \mathbf{s}_3^{-1} \mathbf{s}_2$ is a Mikado braid. Let $y=s_3 s_2$. Then $$N(y)=\{ s_3, s_3 s_2 s_3=s_2 s_3 s_2\}.$$
Let $x:=s_2 s_1 s_3 s_2$. The expression $s_2 s_1 s_3 s_2$ is $S$-reduced. Setting $i_1=2, i_2= 1, i_3=3, i_4=2$ we have $s_{i_4}= s_2$, $s_{i_4} s_{i_3} s_{i_4}= s_3 s_2 s_3$, $s_{i_4} s_{i_3} s_{i_2} s_{i_3} s_{i_4}=s_2 s_1 s_2$, $s_{i_4} s_{i_3} s_{i_2} s_{i_1} s_{i_2} s_{i_3} s_{i_4}=s_3 s_2 s_1 s_2 s_3$. Hence we get $x_{N(y)}=\mathbf{s}_2 \mathbf{s}_1 \mathbf{s}_3^{-1} \mathbf{s}_2=\beta$, which shows that $\beta$ is Mikado.  
\end{exple}

\begin{lem}[Dyer]\label{mik_basic}
Let $x,y\in W$. Then 
\begin{enumerate}
\item We have $(xy^{-1})_{N(y)}= \mathbf{x} \mathbf{y}^{-1}$,
\item If $y=e$, then $x_{N(y)}=\mathbf{x}$. 
\end{enumerate}
\end{lem}
\begin{proof}
For the first item see for instance~\cite[Lemma 3.3 (4)]{Twisted}. The second item is an immediate consequence of the definition of Mikado braids and the fact that $N(e)=\emptyset$. 
\end{proof}

Let $B(W)^+$ be the positive braid monoid. Define a partial order $\leq$ on $B(W)$ by $a\leq b$ if and only if $a^{-1}b\in B(W)^+$. We have the following characterizations of Mikado braids, which are borrowed from \cite[Section 9]{Dyernil}, \cite[Proposition 4.3]{DG}, \cite[Lemma 3.3]{Twisted}: 

\begin{theorem}\label{thm:caract}
Let $\beta\in B(W)$. The following are equivalent
\begin{enumerate}
\item The element $\beta$ is a Mikado braid,
\item There are $x, y\in W$ such that $\beta=\mathbf{x}^{-1} \mathbf{y}$, 
\item There are $x, y\in W$ such that $\beta=\mathbf{x} \mathbf{y}^{-1}$, 
\item We have $\Delta^{-1}\leq \beta \leq \Delta$, where $\Delta$ is the Garside element for the classical Garside structure (i.e., $\Delta$ is the positive lift of the longest element $w_0\in W$). 
\end{enumerate}
\end{theorem}

\begin{exple}
Let $\beta=\mathbf{s}_1 \mathbf{s}_3^{-1} \mathbf{s}_2$ be as in Example~\ref{ex_mik}. We wish to express $\beta$ as a quotient of two positive canonical lifts of elements of the group, as predicted by the Theorem above. Using Example~\ref{ex_mik} and Lemma~\ref{mik_basic} we should have $\beta= \mathbf{x} \mathbf{y}^{-1}$ with $xy^{-1}=s_2 s_1 s_3 s_2$ and $y= s_3 s_2$, hence $x=s_2 s_1 s_3 s_2 s_3 s_2= s_2 s_1 s_2 s_3$. We indeed have $$\mathbf{x} \mathbf{y}^{-1}=\mathbf{s}_2 \mathbf{s}_1 \mathbf{s}_2 \mathbf{s}_3 \mathbf{s}_2^{-1} \mathbf{s}_3^{-1}=\mathbf{s}_2 \mathbf{s}_1 \mathbf{s}_3^{-1} \mathbf{s}_2 \mathbf{s}_3 \mathbf{s}_3^{-1}=\beta.$$ 
\end{exple}

One of the interesting properties of Mikado braids is that their images in the Iwahori-Hecke algebra of the Coxeter system have polynomials with nonnegative coefficients as coordinates when expressed in the Kazhdan-Lusztig basis. The proof of this result requires categorification techniques. This is one of the interesting representation theoretic features of Mikado braids. We briefly recall these properties. 

Let $v$ be an indeterminate. Recall that the Iwahori-Hecke algebra $\mathcal{H}(W)$ of $W$ is the $\mathbb{Z}[v^{\pm 1}]$-algebra generated by $T_s$, $s\in S$, satisfying the braid relations of $W$ and the quadratic relations $T_s^2=(v^{-2}-1) T_s + v^{-2}$ for all $s\in S$. It follows from the quadratic relations that the $T_s$ are invertible in $\mathcal{H}(W)$ and one has therefore a group homomorphism $\varphi: B(W)\longrightarrow \mathcal{H}(W)^\times$ induced by $\mathbf{s}\mapsto T_s$. For $x\in W$, with positive lift $\mathbf{x}\in B(W)$, denote by $T_x$ the element $\varphi(\mathbf{x})\in\mathcal{H}(W)$. It is well-known that $\{T_x\}_{x\in W}$ is a $\mathbb{Z}[v^{\pm 1}]$-linear basis of $\mathcal{H}(W)$, called \textit{standard basis}. Let $\{C_w\}_{w\in W}$ be the Kazhdan-Lusztig $C$-basis (see \cite{Twisted} for the notation and references). One has the following (conjectured by Dyer in~\cite{Dyer_thesis}; see~\cite{DL} for the case of finite Weyl groups, \cite{Twisted} for the general case)

\begin{theorem}\label{thm:gobet}   
Let $W$ be a (not necessarily finite) Coxeter group. Let $x, y\in W$. Then $$\varphi(\mathbf{x} \mathbf{y}^{-1})= T_x T_y^{-1}\in \sum_{w\in W}\mathbb{Z}_{\geq 0}[v^{\pm 1}] C_w.$$
\end{theorem}

Classifying braids which have an image which is positive in the Kazhdan-Lusztig basis appears as an interesting question to us. If $W$ is finite, we have no example of a braid which is not a Mikado braid but satisfies this property.

\subsection{Simple dual braids are Mikado braids}\label{sdbmik}

Let $p: B(W)\twoheadrightarrow W$ denote the canonical surjection induced by $\mathbf{s}\mapsto s$, $s\in S$. For all $x,y\in W$, we have $p(x_{N(y)})=x$. The following Theorem has been (conjectured for all finite $W$ and) shown by Digne and the author \cite{DG} for every irreducible Coxeter system except for type $D_n$. Baumeister and the author \cite{BG} proved it in type $D_n$. Independently, Licata and Queffelec~\cite{LQ} gave a completely different proof in types $A_n$, $D_n$ and $E_n$. In \cite{DG}, \cite{BG}, the proof requires topological realizations of Artin groups (in the classical types): Mikado braids are characterized topologically and the proof uses this characterization. In~\cite{LQ}, the proof uses homological properties of categorifications of Mikado braids (by bounded complexes in a suitable homotopy category). 

Mikado braids indeed seem to have remarkable properties in $2$-braid groups: in \cite{Twisted}, it is shown that the Rouquier complex of a Mikado braid (attached to a non necessarily finite Coxeter group) lies in the heart of the canonical $t$-structure on the bounded homotopy category of Soergel bimodules (Theorem~\ref{thm:gobet} in the general case is a consequence of this property). In \cite{LQ}, Licata and Queffelec show a similar property in a bounded homotopy category of projective modules over a zigzag algebra, in types $A_n$, $D_n$ and $E_n$.  

\begin{theorem}\label{simple_mikado}
Let $W$ be a finite Coxeter group, $c\in\Std(W,S)$, $x\in\NC(W,c)$. The simple dual braid $x_c\in B_c^*$ is a Mikado braid. 
\end{theorem}

An immediate consequence of this Theorem and Theorems~\ref{thm:gobet} and \ref{thm:caract} is: 

\begin{cor}\label{dual_kl_positive}
Let $W$ be a finite Coxeter group, $c\in\Std(W,S)$, $x\in\NC(W,c)$. Then $$\varphi(x_c)\in\sum_{w\in W}\mathbb{Z}[v^{\pm 1}] C_w.$$
\end{cor}


\subsection{A closed formula for simple dual braids}

Expressing simple dual braids in the classical Artin generators is a fundamental problem. In light of Theorem~\ref{simple_mikado}, given $x\in\NC(W,c)$, there is $y\in W$ such that the simple dual braid $x_c$ is equal to $x_{N(y)}$. In principle, expressing simple dual braids in the classical Artin generators requires an inductive application of dual braid relations, and no closed formula is known (some of the proofs of Theorem~\ref{simple_mikado} mentioned in Section~\ref{sdbmik} provide an algorithm to write a simple dual braid as a Mikado braid, but no closed formula). In general, there might be several $y$ such that $x_c=x_{N(y)}$, but it is natural to wonder if there is a natural one. The following Theorem, which is the main result of this paper, gives an answer to this question, and also provides a new (entirely combinatorial proof) of Theorem~\ref{simple_mikado}:

\begin{theorem}\label{thm:main}
Let $W$ be a finite Coxeter group, $c\in\Std(W,S)$, $x\in\NC(W,c)$. Let $y=x^{-1}c$ be the right Kreweras complement of $x$. Then $x_c= x_{N(S_c(y))}$. 
\end{theorem}

In other words, the suitable element by which one has to twist $x$ to get the simple dual braid $x_c$ is the $c$-sortable element corresponding to the Kreweras complement $y$ of $x$. Note that it also shows that $x_c$ is actually obtained by lifting an (in fact, any) $S$-reduced expression of $x$, hence that the length of $x_c$ in $B(W)$ (with respect to $\mathbf{S}\cup \mathbf{S}^{-1}$) is equal to the length of its image $x=p(x_c)$ in $W$, which is not clear at all from the definition of dual braid monoids. 

The rest of the section is devoted to proving Theorem~\ref{thm:main}, assuming Lemma~\ref{lem:casparcasbis} (which will be proven in Section~\ref{explicit} using a case-by-case argument). The strategy is to take advantage of the nice properties of $c$-sortable elements given in Lemmatas~\ref{lem:read1} and \ref{lem:read2}. That is, the proof is by induction on the classical length of $S_c(y)$ and the rank of $W$, where we allow $c$ to vary (i.e., the statement is proven simultaneously for \textit{all} Coxeter elements). We start by proving the following Lemma: 

\begin{lem}\label{lem:camb}
Let $x\in \NC(W,c)$. Let $y=x^{-1}c$. Let $t\in T$ such that $t\leq_T x$. Then $t\notin N( S_c(y))$. 
\end{lem}

\begin{proof}
We argue by induction on the classical length of $S_c(y)$ and the rank of $W$. Let $s$ be initial in $c$. If $W$ has rank $0$ or $1$ then the claim is obvious. 

Firstly, assume that $s S_c(y) > S_c(y)$. By Lemma~\ref{lem:read1}, it implies that $S_c(y)\in W_{<s>}$ and that $S_c(y)$ is $sc$-sortable. It follows that $sxy=sc$ with $\lt(sx)+\lt(y)=\lt(sc)$. If $t\leq_T sx$ then the claim holds by induction on the rank of $W$. Hence assume that $t\not\leq_T sx$. Since $t\leq_T x$, we have that $t$ fixes $V^x$ by Theorem~\ref{carter}. Assume that $t\in W_{<s>}$. Then $t$ fixes $L:=V^{W_{<s>}}$, which is a line. We have that $\dim (V^x)=\dim(V^{sx})-1$. We claim that $L\cap V^x=0$. Otherwise, we have $L\subseteq V^x= V^{sx}\cap V^s$, implying that $s$ fixes $L=V^{W_{<s>}}$, a contradiction because $L$ would be fixed by every element in $W$. Since $sx$ fixes both $L$ and $V^x$, we have that $V^{sx}=L\oplus V^x$. It follows that $t$ fixes $V^{sx}$, hence by Theorem~\ref{carter} again that $t\leq_T sx$, a contradiction. Hence $t\notin W_{<s>}$ which implies that $t\notin N(S_c(y))$ (because $S_c(y)\in W_{<s>}$).   

Now assume that $s S_c(y) < S_c(y)$. We first consider the case where $s$ is a cover reflection of $S_c(y)$. It implies that $s S_c(y)$ is $scs$-sortable and $S_{scs}^{-1} (s S_c(y))= ys$ (by Lemma~\ref{lem:read2} and \cite[Lemma 6.5]{Reading}). By induction on length, we have that for every $q\leq_T sx$, $$q\notin N( S_{scs} (ys))=N(s S_c(y))= N(S_c(y))\backslash\{ s\},$$ where the last equality holds because $s$ is a cover reflection of $S_c(y)$. Since $s$ is a cover reflection of $S_c(y)$ we have $s\leq_T y$, hence $ys\leq_T y$ and $x\leq_T sx$. It implies that $t\leq_T sx$, hence (as we have seen above) that $t\notin N( S_c(y)) \backslash\{ s\}$. Since $t\leq_T x$ and $s\not\leq_T x$ we have $t\neq s$. Hence $t\notin N( S_c(y))$. 

We now assume that $s$ is not a cover reflection of $S_c(y)$. Then $s S_c(y)$ is $scs$-sortable and $S_{scs}^{-1}(s S_c(y))=sys$ (by Lemma~\ref{lem:read2} and \cite[Lemma 6.5]{Reading}). We have $scs=(sxs)(sys)$ and $\ell_T(sxs)+\lt(sys)=\lt(scs)$. By induction on length, we have that for every $q\leq_T sxs$, $$q\notin N(s S_c(y))= s N( S_c(y)) s\backslash\{ s\}.$$

We have $t\leq_T x$, hence $sts\leq_T sxs$, hence $sts\notin s N(S_c(y))s \backslash\{s\}$. It implies that either $t\notin N( S_c(y))$, in which case we are done, or $t=s$. But since $t\leq_T x$, if $t=s$ we have that $sx\in W_{<s>}$ (see \cite[Lemma 5.2]{Reading} or use Carter's Lemma) and $sxy=sc$, implying that $y\in W_{<s>}$. In particular, $S_c(y)$ has all its cover reflections in $W_{<s>}$, implying that $S_c(y)\in W_{< s >}$ by \cite[Lemma 2.24]{RS}. This contradicts $s S_c(y) < S_c(y)$. Hence $t\neq s$, and $t\notin N(S_c(y))$.  
\end{proof}

\subsection{Proof of the main theorem}

\begin{proof}[Proof of Theorem~\ref{thm:main}]
We argue by Cambrian recurrence, i.e., by induction on the rank and length of $S_c(y)$ (as in the proof of Lemma~\ref{lem:camb}). If $\ell(y)=0$, then $x=c$, hence without assumption on the rank of $W$ we have that $x_{N(S_c(y))}=x_\emptyset=c_\emptyset=c_c$ (note that the case where the rank of $W$ is zero is enough to start the recurrence). Assume that $\ell(S_c(y))>0$. Let $s$ be initial in $c$.\\ 

\noindent\textbf{First case: $s S_c(y) > S_c(y)$.} By Lemma~\ref{lem:read1}, it implies that $S_c(y)\in W_{<s>}$ and that $S_c(y)$ is $sc$-sortable. We have $\Read_{sc}(S_c(y))=y\in W_{<s>}$ (hence $S_{sc}(y)=S_c(y)$) and therefore $sc=(sx) y$ with $\ell_T(sx)+\ell_T(y)=\ell_T(sc)$. By induction on the rank of $W$ we get that $$(sx)_{sc}=(sx)_{N(S_{sc}(y))}.$$ Note that $S_{sc}(y)=S_c(y)$. Since $sx\in W_{<s>}$ we have $(sx)_{sc}=(sx)_c$ by Lemma~\ref{dual1} (2). We have $sx \leq_T x$ which implies that $x_c=s_c (sx)_c= \mathbf{s} (sx)_c$. Hence to prove that $x_c=x_{N(S_c(y)}$, it suffices to show that $x_{N(S_c(y))}= \mathbf{s}(sx)_{N(S_c(y))}$. Since $sx\in W_{<s>}$ we have $x > sx$. Choosing a reduced expression of $x$ beginning by $s$ we get that $x_{N(S_c(y))}= \mathbf{s}^{\varepsilon} (sx)_{N(S_c(y))}$, where $\varepsilon=1$ if and only if $x^{-1} s x\notin N(S_c(y))$. But since $sx\in W_{<s>}$, we have $x^{-1} sx\notin W_{<s>}$, hence $x^{-1}sx\notin N(S_c(y))$ since $S_c(y)\in W_{<s>}$. This concludes in this case.\\

\noindent\textbf{Second case: $s S_c(y) < S_c(y)$ and $s$ is a cover reflection of $S_c(y)$.} By Lemma~\ref{lem:read2}, it implies that $s S_c(y)$ is $scs$-sortable. We have that $S_{scs}^{-1}(s S_c(y))= ys \leq_T y$ (see \cite[Lemma 6.5]{Reading} and its proof). We have $scs=(sx)(ys)$ with $\ell_T(scs)=\ell_T(sx)+\ell_T(ys)$. Note that $x\leq_T sx\leq_T scs$ since $ys\leq_T y$. We then have $sxs\leq_T xs\leq_T c$ which, by~\cite[Lemma~5.2]{Reading}, implies that $sxs\in W_{<s>}$. By induction on length we have that 
\begin{equation}\label{eq_1}
(sx)_{N(s S_c(y))}=(sx)_{N(S_{scs}(ys))}=(sx)_{scs}.
\end{equation}
We have to show that $x_{N(S_c(y))}=x_c$. Since $s$ is initial in $c$ and $sxs\in W_{<s>}$, by Lemmatas~\ref{dual1} and \ref{dual2} we obtain 
\begin{equation}\label{eq_2}
\mathbf{s}^{-1} x_c \mathbf{s}=(sxs)_{scs}=(sxs)_{sc}=(sxs)_c.
\end{equation}
Moreover, since $sxs\leq_T sx\leq_T scs$ we have that $(sx)_{scs}=(sxs)_{scs} \mathbf{s}$. Recall that $sxs\in W_{<s>}$, hence $sxs< sx, xs$. If $x < sx$ or $x<xs$, then $x=sxs$ by the exchange lemma. 

Firstly, assume that $x<sx$ (and hence $x=sxs$). Note that since $s$ is a cover reflection of $S_c(y)$, we have $N(s S_c(y))= N(S_c(y))\backslash\{ s\}$. It follows that $(sx)_{N(s S_c(y))}=(xs)_{N(s S_c(y))}= x_{N(S_c(y))}\mathbf{s}$. This concludes because in that case using also \eqref{eq_1} and \eqref{eq_2} we have $$x_c=(sxs)_c=(sxs)_{scs}=(sx)_{scs} \mathbf{s}^{-1}=(sx)_{N(s S_c(y))}\mathbf{s}^{-1}=x_{N(S_c(y))}.$$

Now assume that $sx, xs < x$. We have that $x\leq_T sx$, hence $x^{-1} s x\leq_T sx$. By Lemma~\ref{lem:camb} it implies that $x^{-1} s x\notin N(s S_c(y))$ (we apply the Lemma with $scs=(sx)(ys)$). It follows that $\mathbf{s} (sx)_{N(s S_c(y))}=x_{N(s S_c(y)}$. Lemma~\ref{dual2} then yields
\begin{equation}\label{eq_3}
x_c= \mathbf{s} (sxs)_{scs}\mathbf{s}^{-1}=\mathbf{s} (sx)_{scs} \mathbf{s}^{-1} \mathbf{s}^{-1}.
\end{equation}
By induction and \eqref{eq_3} we get $$x_c=x_{N(s S_c(y))} \mathbf{s}^{-2}=(xs)_{N(S_c(y))} \mathbf{s} \mathbf{s}^{-2}= (xs)_{N(S_c(y)} \mathbf{s}^{-1}= x_{N(S_c(y))},$$ where in the second and last equality we used that $xs < x$ and $N(sS_c(y))=N(S_c(y))\backslash\{ s\}$ (which implies that $t\in N(S_c(y))\backslash\{s\} \Leftrightarrow sts \in N(S_c(y))\backslash\{s\}$). This concludes in that case.\\

\noindent\textbf{Third case: $s S_c(y) < S_c(y)$ and $s$ is not a cover reflection of $S_c(y)$.} It implies that $s S_c(y)$ is $scs$-sortable by Lemma~\ref{lem:read2}. We have that $S_{scs}^{-1}(s S_c(y))= sys$ (see~\cite[Lemma 6.5]{Reading}). Since $(sxs)_{scs}=\mathbf{s}^{-1} x_c \mathbf{s}$, by induction on the length it suffices to show that $x_{N(S_c(y))}= \mathbf{s} (sxs)_{N(s S_c(y))} \mathbf{s}^{-1}$. 

Note that $(sxs)_{N(sS_c(y))}=(sx)_{N(S_c(y))}\mathbf{s}$ (this can be seen by a direct computation using the definition of Mikado braids and distinguishing the cases $sx<sxs$ and $sxs<sx$). Hence it suffices to show that $\mathbf{s}(sx)_{N(S_c(y))}=x_{N(S_c(y))}$. If $sx <x$, then $s\in N(x)$ and by Lemma~\ref{lem:casparcasbis} together with Remark~\ref{rmq:refor}  we have that $x^{-1} sx\notin N(S_c(y))$. It follows that $x_{N(S_c(y))}=\mathbf{s} (sx)_{N(S_c(y))}$. If $sx >x$, then $s\notin N(x)$ and by Lemma~\ref{lem:casparcasbis} and Remark~\ref{rmq:refor} we have that $x^{-1} sx\in N(S_c(y))$. It follows that $(sx)_{N(S_c(y))}= \mathbf{s}^{-1}x_{N(S_c(y))}$. Hence in both cases ($x<sx$ and $x>sx$) we get that $\mathbf{s}(sx)_{N(S_c(y))}=x_{N(S_c(y))}$, which concludes in this case.

This completes the proof of Theorem~\ref{thm:main}. 
\end{proof}

\section{Bijections between noncrossing partitions and $c$-sortable elements}
\label{explicit}

The aim of this section is to give an explicit description of the inverse $S_c$ of Reading's bijection $\Read$ from noncrossing partitions to $c$-sortable elements (in the classical and dihedral types), and use it to prove Lemma~\ref{lem:casparcasbis}. We study each type in a different section. Types $A_n$ and $D_n$ are treated using the noncrossing partition models, while for type $B_n$ one can use the realization of the Coxeter group as the subgroup of a Coxeter group $W'$ of type $A_{2n-1}$ consisting of those elements which are fixed by the automorphism of $W'$ induced by the automorphism of the type $A_{2n-1}$ Dynkin diagram.  

\subsection{Type $A_n$}

\subsubsection{Graphical noncrossing partitions}\label{graph_a}

We identify $W$ with the symmetric group $\mathfrak{S}_{n+1}$ and $S$ with the set of simple transpositions. We denote by $s_i$ the simple transposition $(i, i+1)$, $i=1,\dots, n$. Let $c\in\Std(\Sn)$. The elements of $\NC(\Sn,c)$ can be represented by disjoint unions of polygons with vertices on a circle labeled by $[n+1]:=\{1,\dots,n+1\}$. Every Coxeter element is an $(n+1)$-cycle. For $c=(i_1, i_2,\dots, i_{n+1})$, the labeling of the circle is given (in clockwise order) by $i_1 i_2\cdots i_{n+1}$. We call it the \defn{c-labeling}. Assuming that $i_1=1$, $i_k=n+1$, one has that $c$ lies in $\Std(\Sn)$ if and only if the sequence $i_1i_2\cdots i_k$ is increasing and the sequence $i_k\cdots i_{n} i_{n+1}$ is decreasing (see for instance~\cite[Lemma 8.2]{GobWil}). We set $R_c:=\{i_1, i_2, \dots, i_k\}$, $L_c:=\{ i_k, i_{k+1}, \dots, i_{n+1}, i_1\}$. An example is given in Figure~\ref{figure:orientation} below. To every cycle in the decomposition of $x\in\NC(\Sn,c)$ into a product of disjoint cycles, one associates the polygon obtained as the convex hull of the set of points lying in the support of the cycle (we identify the points with their labels; polygons can be reduced to a line segment or a point). We denote the set of polygons of $x$ by $\Pol(x)$ and identify a polygon with the set of integers labeling its vertices. In the obtained diagram, all the polygons are disjoint, equivalently the partition defined by the cycle decomposition of $x$ is \textit{noncrossing} for the $c$-labeling (see~\cite{Biane} for more details). A reflection $(i,j)$ lies in the parabolic subgroup $P(x)$ if and only if the line segment $ij$ is an edge or a diagonal of a polygon $P\in\Pol(x)$, that is, if and only if there is $P\in \Pol(x)$ with $i,j\in P$, while it is a canonical generator of $P(x)$ if and only if $i,j\in P$ with no $i<k<j$ such that $k\in P$.   

It turns out to be convenient for some proofs to slightly modify the graphical representation in the following way: let $c=(i_1,\dots, i_{n+1})\in\Std(\Sn)$ with $i_1=1$, $i_k=n+1$. Instead of drawing the points on the circle such that the length of the segment between any two successive points is constant, we draw the point with label $1$ at the top of the circle, the point with label $n+1$ at the bottom, the points with label in $i_k\cdots i_n i_{n+1}$ on the left and the points with label in $i_1i_2\cdots i_k$ on the right, each point having a specific height depending on its label. More precisely, if $i, j\in[n+1]$, $i<j$, then the point $i$ is higher than the point $j$ (as done in the right picture of Figure~\ref{figure:orientation}). Also, when representing a noncrossing partition we may use curvilinear polygons instead or regular polygons for convenience (as in Figure~\ref{figure:deux}), but with the convention that every edge of a polygon should be either strictly increasing or decreasing. 



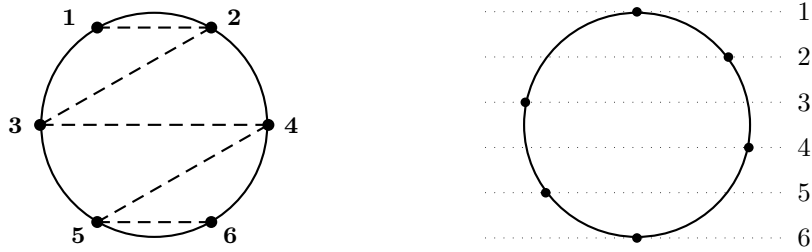
\begin{figure}[h!]
\begin{center}
\begin{tabular}{ccc}
& \begin{pspicture}(1,0)(6,3)
\pscircle(3,1.5){1.5}
\psdots(4.5,1.5)(3.75,2.79)(3.75,.21)(2.25,.21)(2.25,2.79)(1.5,1.5)
\psline[linestyle=dashed](2.25,2.79)(3.75,2.79)
\psline[linestyle=dashed](1.5,1.5)(3.75,2.79)
\psline[linestyle=dashed](1.5,1.5)(4.5,1.5)
\psline[linestyle=dashed](4.5,1.5)(2.25,.21)
\psline[linestyle=dashed](2.25,.21)(3.75,.21)
\psdots[dotsize=4.5pt](3.75,.21)(2.25,.21)
\psdots[dotsize=4.5pt](1.5,1.5)(2.25,2.79)(4.5,1.5)(3.75,2.79)
\rput(4.05,2.92){\textrm{{\footnotesize \textbf{2}}}}
\rput(4.8,1.5){\textrm{{\footnotesize \textbf{4}}}}
\rput(4,.03){\textrm{{\footnotesize \textbf{6}}}}
\rput(2,.03){\textrm{{\footnotesize \textbf{5}}}}
\rput(1.17,1.5){\textrm{{\footnotesize \textbf{3}}}}
\rput(1.88,2.92){\textrm{{\footnotesize \textbf{1}}}}
\end{pspicture} & \begin{pspicture}(0,0)(6,3)
\pscircle(3,1.5){1.5}

\psline[linestyle=dotted, linewidth=0.4pt](1,3)(5,3)
\psline[linestyle=dotted, linewidth=0.4pt](1,2.4)(5,2.4)
\psline[linestyle=dotted, linewidth=0.4pt](1,1.8)(5,1.8)
\psline[linestyle=dotted, linewidth=0.4pt](1,1.2)(5,1.2)
\psline[linestyle=dotted, linewidth=0.4pt](1,0.6)(5,0.6)
\psline[linestyle=dotted, linewidth=0.4pt](1,0)(5,0)

\psdots(3,3)(4.2, 2.4)(1.5304, 1.8)(4.4696, 1.2)(1.8, 0.6)(3,0)
\rput(5.2,3){\small $1$}
\rput(5.2,2.4){\small $2$}
\rput(5.2,1.8){\small $3$}
\rput(5.2,1.2){\small $4$}
\rput(5.2,0.6){\small $5$}
\rput(5.2,0){\small $6$}
\end{pspicture}
\end{tabular}
\end{center}
\caption{The $c$-labeling for $c=s_1 s_3 s_5 s_4 s_2=(1,2,4,6,5,3)\in\Std(\Sn)$. We have $L_c=\{1,3,5,6\}$, $R_c=\{1,2,4,6\}$.}
\label{figure:orientation}
\end{figure}


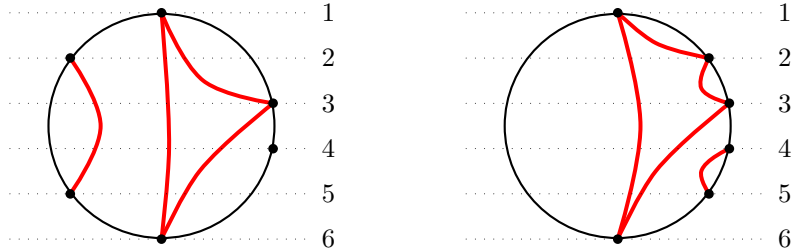
\begin{figure}[h!]
\psscalebox{1}{
\begin{pspicture}(6.3,0)(6,3)
\pscircle(3,1.5){1.5}
\pscurve[linecolor=red, linewidth=1.5pt](3,3)(3.55, 2.1)(4.4696, 1.8)
\pscurve[linecolor=red, linewidth=1.5pt](3,0)(3.1,1.2)(3,3)
\pscurve[linecolor=red, linewidth=1.5pt](3,0)(3.5,0.9)(4.4696, 1.8)
\pscurve[linecolor=red, linewidth=1.5pt](1.8, 2.4)(2.2,1.5)(1.8, 0.6)

\psline[linestyle=dotted, linewidth=0.4pt](1,3)(5,3)
\psline[linestyle=dotted, linewidth=0.4pt](1,2.4)(5,2.4)
\psline[linestyle=dotted, linewidth=0.4pt](1,1.8)(5,1.8)
\psline[linestyle=dotted, linewidth=0.4pt](1,1.2)(5,1.2)
\psline[linestyle=dotted, linewidth=0.4pt](1,0.6)(5,0.6)
\psline[linestyle=dotted, linewidth=0.4pt](1,0)(5,0)

\psdots(3,3)(1.8, 2.4)(4.4696, 1.8)(4.4696, 1.2)(1.8, 0.6)(3,0)
\rput(5.2,3){\small $1$}
\rput(5.2,2.4){\small $2$}
\rput(5.2,1.8){\small $3$}
\rput(5.2,1.2){\small $4$}
\rput(5.2,0.6){\small $5$}
\rput(5.2,0){\small $6$}
\pscircle(9,1.5){1.5}
\pscurve[linecolor=red, linewidth=1.5pt](9,3)(9.5, 2.6)(10.2, 2.4)
\pscurve[linecolor=red, linewidth=1.5pt](10.2, 2.4)(10.1,2)(10.4696, 1.8)
\pscurve[linecolor=red, linewidth=1.5pt](10.4696, 1.8)(9.5,0.9)(9,0)
\pscurve[linecolor=red, linewidth=1.5pt](9,3)(9.3,1.5)(9,0)

\pscurve[linecolor=red, linewidth=1.5pt](10.4696, 1.2)(10.1,0.9)(10.2, 0.6)

\psline[linestyle=dotted, linewidth=0.4pt](7,3)(11,3)
\psline[linestyle=dotted, linewidth=0.4pt](7,2.4)(11,2.4)
\psline[linestyle=dotted, linewidth=0.4pt](7,1.8)(11,1.8)
\psline[linestyle=dotted, linewidth=0.4pt](7,1.2)(11,1.2)
\psline[linestyle=dotted, linewidth=0.4pt](7,0.6)(11,0.6)
\psline[linestyle=dotted, linewidth=0.4pt](7,0)(11,0)

\rput(11.2,3){\small $1$}
\rput(11.2,2.4){\small $2$}
\rput(11.2,1.8){\small $3$}
\rput(11.2,1.2){\small $4$}
\rput(11.2,0.6){\small $5$}
\rput(11.2,0){\small $6$}

\psdots(9,3)(10.2, 2.4)(10.4696, 1.8)(10.4696, 1.2)(10.2, 0.6)(9,0)\end{pspicture}}
\caption{Conventions on the graphical representations of $c$-noncrossing partitions. On the left: $(1,3,6)(2,5)\in\NC(\mathfrak{S}_6, s_2s_1s_3s_5s_4)$; on the right: $(1,2,3,6)(4,5)\in\NC(\mathfrak{S}_6,s_1 s_2s_3s_4s_5)$.}
\label{figure:deux} 
\end{figure}

\subsubsection{Explicit description of the map from noncrossing partitions to $c$-sortable elements}\label{expl_a}

\begin{definition}[Relative positions of polygons]\label{def_rel}
Let $P, Q$ be two distinct polygons of $x$. We say that $P$ is \defn{strictly at the left of $Q$} if there is a horizontal line $L$ intersecting both $Q$ and $P$ and $L\cap P$ is at the left of $L\cap Q$. In this situation we also say that $Q$ is \defn{strictly at the right of $P$}. There might be several horizontal lines intersecting both $P$ and $Q$ non trivially, but it is clear that the intersections with horizontal lines will always appear in the same order. We say that $P$ and $Q$ are \defn{fully disjoint} if none of them lies strictly at the right of the other. We say that $P$ is \defn{at the left of $Q$} if either it is strictly at the left of $Q$, or they are fully disjoint (and similarly $Q$ lies \defn{at the right of $P$}). Note that given any two distinct polygons, there is always one of them lying at the left of the other. 
\end{definition}

Note that determining the relative position of two polygons can be done by just looking at the relative position of the two diagonals joining the smallest to the biggest index of each polygon, instead of considering the whole polygons. 

We define a total order $\prec$ on the set $\Pol(x)$ of polygons of $x\in\NC(\Sn, c)$. In the graphical representation of $x$, let $P\in\Pol(x)$ be the unique polygon such that
\begin{enumerate}
\item $P$ has no polygon lying strictly at its right  (i.e., any horizontal line crossing $P$ crosses no polygon at the right of $P$),
\item $\max(P)$ is minimal among all polygons satisfying property (1). 
\end{enumerate}

Then $P$ appears first in the total order. We then remove $P$ (together with all its vertices) from the graphical representation. The second polygon $Q$ in the total order is the unique polygon among the remaining ones such that there is no polygon (among the remaining ones) lying strictly at the right of $Q$ and $\max(Q)$ is minimal for this property, and so on. 

We use this order to define a permutation $S_c(x)\in\Sn$. If $P$ is the first polygon in the order and $P=\{ d_1 > d_{2} > \dots > d_k\}$, we define $S_c(x)$ on the subset $\{1, 2, \dots, k\}$ by $1\mapsto d_1$, $2\mapsto d_2, \dots, k\mapsto d_k$. If $Q$ is the second polygon in the order and $Q=\{ e_1 > e_2 >\dots > e_\ell\}$, we define $S_c(x)$ on the subset $\{ k+1, \dots, k+\ell\}$ by $k+1\mapsto e_1$, $k+2\mapsto e_2, \dots, k+\ell \mapsto e_\ell\}$, and so on. 

\begin{exple}
Let $c=s_2 s_1 s_3 s_5 s_4$. Let $(1,3,6)(2,5)\in\NC(\mathfrak{S}_6, c)$ (the graphical representation is given in the left picture of Figure~\ref{figure:deux}). The only polygon with no other polygon strictly at its right is $\{4\}$. After removing it, the only polygon with no other polygon strictly at its right is $\{6,3,1\}$. Hence the total order on the set of polygons of $x$ is $$\{4\} \prec \{6,3,1\} \prec \{5,2\}.$$
It follows that $$S_c(x)=(1,4)(2,6)=s_2 s_1 s_3 s_5 s_4|s_2 s_1 s_3 s_5| s_2.$$
The reduced expression above shows that $S_c(x)$ is $c$-sortable. This is a general fact, proven in Theorem~\ref{a} below.  
\end{exple} 

Note that if $Q$ is strictly at the right of $P$, then we necessarily have $Q\prec P$. It can happen that $P$ and $Q$ are fully disjoint, with $j<k$ for all $j\in P$ and $k\in Q$, but nevertheless $Q\prec P$. This is, me may have a sequence $Q_1, \dots, Q_m$ of pairwise distinct polygons of $x$ with $Q$ strictly at the right of $Q_1$, $Q_1$ strictly at the right of $Q_2$, $\dots$, $Q_m$ strictly at the right of $P$, but $P$ and $Q$ fully disjoint. 

\begin{theorem}\label{a}
Let $x\in\NC(\Sn, c)$. The element $S_c(x)\in\Sn$ is $c$-sortable and we have $\Read(S_c(x))=x$. That is, we have $S_c=\Read^{-1}$. 
\end{theorem}

\begin{proof}
Rather than using the definition of $c$-sortable elements, we make use of their characterization as a permutation: in \cite[Theorem 4.12]{Reading}, Reading shows that a permutation in $\Sn$ is $c$-sortable if and only if its line notation avoids the pattern $jki$ for all $i<j<k$ such that $j\in L_c$  and avoids the pattern $kij$ for all $i<j<k$ such that $j\in R_c$. 

Let $i<j<k$ and assume that $j\in L_c$. Firstly, we assume that $i\in P\in\Pol(x)$ with $P\neq Q$, where $Q\in\Pol(x)$ is the polygon containing $j$. If $P\prec Q$ then in the line notation of $S_c(x)$ we have that $i$ appears before $j$ hence $S_c(x)$ is $jki$-avoiding. Hence assume that $Q\prec P$. If $k\in P$, then the edge $ik$ of $P$ would lie strictly at the right of $Q$ since $j\in Q$ and $i<j<k$, contradicting $Q\prec P$. If $k\in Q$, then by definition of $S_c(x)$ we have that $k$ appears before $j$ in the line notation of $S_c(x)$, hence that $S_c(x)$ is $jik$-avoiding. We can therefore assume that $k\in R\in \Pol(x)$ for some $R\neq P, Q$. If $R\prec Q$ then $k$ appears before $j$, in which case we are done. So we can assume that $Q\prec R$. Since $Q\prec R, P$, it means that after having removed the polygon $Q$ in the order $\prec$, there can be no polygon intersecting the horizontal line $j$ (because $j\in Q\cap L_c$: by definition of $\prec$, when removing $Q$ there can be no polygon strictly at the right of $Q$, whence the claim). It implies that $P\subseteq \{1, \dots, j-1\}$ and $R\subseteq \{j+1, \dots, n+1\}$. Together with the fact that after removing $Q$, no polygon can cross the line $j$, we get that $P\prec R$, which implies that $k$ appears after $i$ in the line notation of $S_c(x)$, hence that it is $jki$-avoiding. 


Now if $P=Q$, then the indices appearing between $j$ and $i$ in the line notation for $S_c(x)$ are all the vertices $m\in P$ such that $i<m<j$, therefore $k$ cannot appear. Hence $S_c(x)$ is $jki$-avoiding. If $j\in R_c$, the proof that $S_c(x)$ is $kij$-avoiding is similar. This proves that $S_c(x)$ is $c$-sortable.

We now prove that $S_c$ is the inverse of $\Read^{-1}$. Recall that the map $\Read$ sends a $c$-sortable element $y$ to the product of its cover reflections in a suitable order. In the line notation $\pi_1 \pi_2\cdots \pi_{n+1}$ of a permutation $y\in\Sn$, the cover reflections are the transpositions $(\pi_i, \pi_{i+1})$ such that $\pi_{i+1} < \pi_i$. By defintion of the element $S_c(x)$, we have that the cover reflections of $S_c(x)$ are precisely the transpositions $(i,j)$ such that $i$ and $j$ are consecutive vertices of a common polygon of $x$: indeed, given a polygon $P=\{ d_1 > d_2 > \cdots  > d_k\}$, there is $r\in [n+1]$ such that $S_c(x)(r+i-1)=d_i$ for all $i=1, \dots, k$, implying that $(d_i, d_{i+1})$ is a cover reflection of $S_c(x)$ for all $i=1, \dots, k-1$. The element $S_c(x)$ has no other cover reflection since if $P \prec Q$ are consecutive in the order which we put on polygons, then $\min P < \max Q$ by definition of the order, in particular in the line notation of $S_c(x)$ the index following $d_k$ has to be bigger than $d_k$. It follows that the cover reflections of $S_c(x)$ are precisely the canonical generators of the parabolic subgroup associated to the noncrossing partition $x$, which using Theorem~\ref{thm_reading} and the remarks preceding Remark~\ref{rmq_centralizer} proves that $x= \Read(S_c(x))$. 

\end{proof}

The fact that $S_c(x)$ is $c$-sortable can also be proven using Lemmatas~\ref{lem:read1} and \ref{lem:read2}. In type $D_n$ we will define $S_c$ similarly and prove this fact using these Lemmatas. 

Before proving Lemma~\ref{lem:casparcasbis} we show

\begin{lem}\label{order_polygons_a}
Let $c$ be a Coxeter element and let $s=s_i$ be initial in $c$ (that is, $i\in L_c$, $i+1\in R_c$). Let $j=c^{-1}(i)$, $k=c^{-1}(i+1)$. Assume that $j\in P_1$, $k\in P_2$, with $P_1\neq P_2$. Then $P_2\prec P_1$.  
\end{lem}

\begin{proof}
If $\ell:=\min(P_1)<k$, then the segment $j\ell$ is strictly at the left of $P_2$ (since $k\in R_c$) implying that $P_2\prec P_1$. Similarly if $\max(P_2) >j$ we have $P_2\prec P_1$. Hence we can assume that $\min(P_1) >k$ and $\max(P_2)<j$. In this situation, if $P_1$ and $P_2$ are not fully disjoint then we necessarily have that $P_2$ is strictly at the right of $P_1$ implying $P_2\prec P_1$ again (looking at Figure~\ref{figure:order_lemma} is helpful to note this fact: it follows from the fact that $\{k+1, \dots, i\}\subseteq L_c$ and $\{i+1, \dots, j-1\}\subseteq R_c$). We can therefore assume that $P_2$ and $P_1$ are fully disjoint. Moreover since all indices in $P_2$ are strictly smaller than those in $P_1$ (because of fully disjointness and $k\in P_2$, $j\in P_1$ with $k<j$), we have $P_1 \prec P_2$ only if there is a sequence $Q_1, \dots, Q_m$ of polygons of $x$ with $Q_1$ strictly at the left of $P_1$, $Q_i$ strictly at the left of $Q_{i-1}$ for all $i$ and $P_2$ strictly at the left of $Q_m$. Inspection of Figure~\ref{figure:order_lemma} again shows that this situation cannot appear.  

\end{proof}

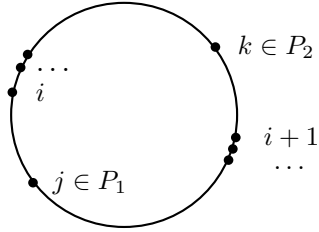
\begin{figure}[h!]
\begin{center}
 \begin{pspicture}(0,0)(6,3)
\pscircle(3,1.5){1.5}


\psdots(4.2, 2.4)(1.5304, 1.8)(4.4696, 1.2)(4.43,1.05)(4.375,0.9)(1.8, 0.6)(1.64, 2.13)(1.73,2.3)

\rput(5,2.4){\small $k\in P_2$}
\rput(1.9,1.8){\small $i$}
\rput(2.1,2.1){$\dots$}
\rput(5.2,1.2){\small $i+1$}
\rput(5.2,0.8){\small $\dots$}
\rput(2.55,0.6){\small $j\in P_1$}
\end{pspicture}
\end{center}
\caption{A picture for the proof of Lemma~\ref{order_polygons_a}. Note that by definition of $j$ and $k$, there are no points on the circle between $j$ and $i$ as well as between $k$ and $i+1$.}
\label{figure:order_lemma}
\end{figure}






We now prove Lemma~\ref{lem:casparcasbis} in type $A_n$:

\begin{proof}[Proof of~\ref{lem:casparcasbis} in type $A_n$]
We reformulate the lemma to drop any reference to $x$. We have $x^{-1} s x= (x^{-1}(i), x^{-1}(i+1))$. Note that since $xy=c$ we have $x^{-1}(i)=y(j)$, $x^{-1}(i+1)=y(k)$, where $j$ precedes $i$ on the circle and $k$ precedes $i+1$. Assume that $x^{-1}(i) < x^{-1}(i+1)$. Since $k<j$ it implies that $x^{-1} s x\in N(y)$, hence we have to show that $x^{-1} s x\in N(S_c(y))$, that is, that $S_c^{-1}(y) (y(j)) > S_c^{-1}(y)(y(k))$. If $x^{-1}(i)> x^{-1} (i+1)$, then $x^{-1} s x\notin N(y)$, hence we have to show that $S_c^{-1}(y(j)) > S_c^{-1}(y(k))$. What we have just proven is that it suffices to show that $$S_c^{-1}(y) (y(j)) > S_c^{-1}(y)(y(k))$$ to establish the Lemma. 

Let $P_1$ be the polygon containing $j$ and $P_2$ be the one containing $k$. Note that $y(j)\in P_1$, $y(k)\in P_2$. If $P_1=P_2$, then $jk$ is a diagonal of $P_1=P_2$, in particular we have $y(j)<y(k)$ (since $j\in L_c$, $k\in R_c$, $k<j$, hence $y(j)\leq i$, $y(k)\geq i+1$). Since $y(j), y(k)\in P_1=P_2$, by definition of $S_c(y)$ it follows that $S_c^{-1}(y(j))> S_c^{-1}(y)(y(k))$, which concludes in this case. Hence assume that $P_1\neq P_2$. Then by Lemma~\ref{order_polygons_a} we have $P_2\prec P_1$, hence $S_c^{-1}(y) (y(j)) > S_c^{-1}(y)(y(k))$ by definition of $S_c$.

\end{proof}

\subsection{Type $B_n$}

Let $W$ be a Coxeter system of type $A_{2n-1}$ and let $\Gamma$ be the automorphism of $W$ induced by the automorphism of the Dynkin diagram which exchanges $s_i$ and $s_{2n-i}$ for all $i=1,\dots, 2n-1$. Then $W^{\Gamma}$ is a Coxeter group of type $B_n$ (see for instance~\cite[Section 6.1]{DG} and the references therein). Its generators as a Coxeter group are given by $t_0=s_n$, $t_1= s_{n-1} s_{n+1}$, $\dots$, $t_{n-1}=s_1 s_{2n-1}$. For convenience we see $W$ as the symmetric group on $\{-n, \dots, -1, 1, \dots, n\}$; in this description we have $t_0=(1, -1)$, $t_i=(i,i+1)(-i,-i-1)$, $\forall~ i\in\{1, \dots, n-1\}$, and $W^\Gamma$ is the subgroup of \defn{signed permutations}, that is, of elements $w\in W$ such that $w(-i)=-w(i)$ for all $i=1,\dots, n$. The set of reflections in $W_{B_n}$ is given by $$T:=T_{B_n}:=\{ (i, -i)~|~i=1, \dots, n\} \cup \{ (i,j)(-i,-j)~|~i\neq \pm j\text{~and~} i, j=1,\dots, n\}.$$ Given $w\in W^{\Gamma}\subseteq W$, we denote by $N^A(w)$ (resp. $N^B(w)$) the inversion set of $w$ in $W$ (resp. $W^{\Gamma}$). For $w\in W^\Gamma$ and $i\in\{1,\dots, n\}$, we have $(i,-i)\in N^A(w)\Leftrightarrow (i,-i)\in N^B(w)$, while for $i, j\in \{1, \dots, n\}$  with $i\neq \pm j$ we have $(i,j)(-i,-j)\in N^B(w)\Leftrightarrow (i,j)\in N^A(w)\Leftrightarrow (-i, -j)\in N^A(w)$ (see for instance~\cite[Section 8.1]{BjBr}). The same equivalences hold if we replace inversion sets by sets of left or right descents and reflections by simple reflections.  

The standard Coxeter elements in $W^\Gamma$ are the Coxeter elements in $W$ which are fixed by $\Gamma$. Let $c$ be such a Coxeter element. Then $\NC(W^\Gamma, c)=\NC(W, c)^\Gamma$, in particular noncrossing partitions of type $B_n$ can be graphically represented by noncrossing partitions of type $A_{2n-1}$ for the $c$-labeling, which are stable by a half turn (see for instance~\cite[Section 6.3]{DG}). 

\begin{lem}\label{fixed_sort}
We have $\Sort(W^\Gamma, c)=\Sort(W, c)^\Gamma$. 
\end{lem}
\begin{proof}
By definition of $c$-sortability, we have that $\Sort(W^\Gamma, c)\subseteq \Sort(W, c)^\Gamma$. Conversely, let $w\in \Sort(W, c)^\Gamma$ and let $w=w_1 w_2 \cdots w_k$, where $w_i$ is the subexpression of the $i$-th copy of $c$ in the $c$-sorting word for $w$. We want to show that $\Gamma(w_i)=w_i$ for all $i$. By induction on $k$ it suffices to show that $\Gamma(w_1)=w_1$. if $\Gamma(w_1)\neq w_1$, then there is an $i$ such that $s_i$ appears in $w_1$ but $s_{2n-i}$ does not (recall that each simple reflection can occur at most once in $w_1$). By definition of $c$-sortability, it implies that $s_{2n-i}$ appears in no $w_i$, hence that there is a reduced expression of $w$ in which $s_i$ appears but $s_{2n-i}$ does not, contradicting $\Gamma(w)=w$ (since the simple reflections occuring in any two reduced decompositions are the same). Hence $\Gamma(w_i)=w_i$ for all $i$. It follows that $w\in \Sort(W^{\Gamma}, c)$. Hence $\Sort(W^\Gamma, c)=\Sort(W, c)^\Gamma$. 
\end{proof}

Now the above equivalences on descent and inversion sets of types $A$ and $B$ imply that for $w\in\Sort( W^\Gamma, c)$ we have $\Read^A(w)=\Read^B(w)$, where $\Read^A$ (resp. $\Read^B$) denotes the map from $c$-sortable elements to noncrossing partitions in type $A$ (respectively in type $B$). As a consequence, the map $S_c^B: \NC(W^{\Gamma}, c)\longrightarrow \Sort(W^\Gamma, c)$ defined as the inverse of $\Read^B$ coincides with the restriction of the inverse $S_c^A=S_c: \NC(W, c)\longrightarrow \Sort(W, c)$ of $\Read^A$ (described in Section~\ref{expl_a}) to $W^\Gamma$.  

It is now easy to derive a proof of Lemma~\ref{lem:casparcasbis} in type $B_n$ from the Lemma in type $A_n$:

\begin{proof}[Proof of Lemma~\ref{lem:casparcasbis} in type $B_n$]
Assume that $s=t_0$. Then $x^{-1}sx=(i,-i)$ for some $i\in\{1,\dots, n\}$. We have that $s$ is initial in $c$ viewed as a Coxeter element of type $A_{2n-1}$. By the properties of inversion sets given above and Lemma~\ref{lem:casparcasbis} in type $A_n$ we have
$$x^{-1} s x\in N^B(y) \Leftrightarrow x^{-1} sx\in N^A(y)\Leftrightarrow x^{-1}sx\in N^A(S_c(y)) \Leftrightarrow x^{-1} sx\in N^B(S_c(y)),$$
which concludes the proof in that case. Now assume that $s=t_i=s_{n-i} s_{n+i}$, $i\in\{1, \dots, n-1\}$. We have that both $s_{n-i}$ and $s_{n+i}$ are initial in $c$ viewed as a Coxeter element of type $A_{2n-1}$. By the properties of inversion sets given above and Lemma~\ref{lem:casparcasbis} in type $A_n$ we have 
$$x^{-1}sx\in N^B(y)\Leftrightarrow x^{-1} s_{n-i} x\in N^A(y) \Leftrightarrow x^{-1} s_{n-i} x\in N^A(S_c(y))\Leftrightarrow x^{-1} sx\in N^B(S_c(y)),$$ 
which concludes the proof in that case. 
\end{proof}

\subsection{Type $D_n$}

\subsubsection{Graphical representations}

The Coxeter group $W_{D_n}$ of type $D_n$ is realized as an index two subgroup of the signed permutation group $W_{B_n}$; a permutation $w\in W_{B_n}$ is in $W_{D_n}$ if and only if $|\{i\in\{1,2,\dots,n\}~|~w(i)<0\}$ is even. The Coxeter generating set is given by $s_0=(1,-2)(-1,2)$, $s_i=(i,i+1)(-i,-i-1)$ for all $i=1,\dots, n-1$. Recall that that $w(-i)=-w(i)$, for all $i\in\{1,\dots, n\}$ and all $w \in W_{B_n}$.
In $W_{B_n}$, cycles of the shape $(i_1, \dots, i_r, -i_1, \dots, -i_r)$ are abbreviated by $[i_1, \dots i_r]$ and called \defn{balanced cycles}, and those
of type $(i_1, \dots, i_r)( -i_1, \dots, -i_r)$ by $((i_1, \dots, i_r ))$ and called \defn{paired cycles}. Every $w \in W_{D_n}$ can be written as a product of disjoint cycles in which there is an even number of balanced cycles (see~\cite[Section 2]{AR}). 

The set of reflections in $W_{D_n}$ is 
$$T:= T_{D_n} :=\{ (i,j)(-i,-j) \mid i, j\in\{-n, \dots, n\}, i\neq \pm j\}.$$The standard Coxeter elements of type $D_n$ are of the form $$(i_1, -i_1)(i_2, \dots, i_n, -i_2, \dots, -i_n),$$ where $\{i_1, \dots, i_n\}=\{1,\dots, n\}$, $i_1\in\{1,2\}$ and the sequence $i_2\cdots i_n$ is first increasing, then decreasing (see \cite[Lemma 4.3]{BG}). If $i_k=n$ we set $$L_c=\{i_k, i_{k+1}, \dots, i_n, -i_2, \dots, -i_k\},~R_c=-L_c.$$ 
The elements of $\NC(W_{D_n}, c)$ also admit noncrossing-like graphical representations. The reference for this is mostly~\cite{AR}, but our conventions are those from~\cite[Section 5]{BG}. We briefly recall how it works and refer the reader to~\cite{BG} for more details. The circle is labelled as in type $A_n$ with the sequence $i_2,\dots, i_n, -i_2, \dots, -i_n$ but there is one additional point at the center of the circle, with the two labels $\pm i_1$. As in type $A_n$, we put the points at different heights such that point $i$ is higher than point $j$ if and only if $i <j$, with $-n$ at the top of the circle and $n$ at the bottom. Then given $x\in \NC(W_{D_n},c)$, we consider the cycle decomposition of $x$ inside $W_{B_n}\subseteq W_{A_{2n-1}}$ and associate to every (type $A_{2n-1}$) cycle a polygon, given by the convex hull of those points having their label in the support of the cycle. If the cycle $(i_1, -i_1)$ appears in the decomposition, in fact in that case there must be a polygon $P$ such that $P=-P$, hence with the middle point labeled by $\pm i_1$ inside it. In that case we leave the middle point untouched. We often identify polygons with their labels; in the case where $P=-P$, consider $\pm i_1$ as lying in $P$ (that is, in the case where we have two balanced cycles in $x$, then one must be $(i_1, -i_1)$; we then consider theses two cycles as a single polygon). We denote the set of polygons of $x$ by $\Pol(x)$. Hence to each paired cycles are associated two polygons $P, -P$, while if there are balanced cycles (in which case there are exactly two of them), we consider them as a single polygon $Q=-Q$ (and we call such a polygon \defn{symmetric}). 
  
This gives a noncrossing diagram, i.e. the obtained polygons are all disjoint, except possibly in two cases:
\begin{itemize}

\item The first case is if there is a polygon $P\in\Pol(x)$ with $i_1\in P$ but $-i_1\notin P$. Since $x$ is a signed permutation, we have that $-P$ is also a polygon of $x$. In that case the polygons $P$ and $-P$ have the middle point in common. Note that we have to specify, in the graphical representation, whether $i_1\in P$ or $i_1\in -P$. 
\item The second case appears if $x$ has a product of cycles of the form $$(j, -j)(i_1,-i_1)$$ for some $j\neq \pm i_1$. In that case, to avoid confusion with the graphical representation of the noncrossing partition $((j, i_1))$ (or $((j, -i_1))$), we join the point $j$ to the point $-j$ by two curves in such a way that the middle point lies between the two curves (the cycle $(j,-j)$ corresponds to a polygon $P$ with $P=-P$, hence this way of representing the situation should be seen as a special case of the situation mentioned above when a factor $(i_1, -i_1)$ appears in the decomposition of $x$). 

\end{itemize}

Conversely, every noncrossing diagram as above can be used to define an element of $\NC(W_{D_n}, c)$, by associating to a polygon $P\neq -P$, $P=\{m_1, \dots, m_k\}$ (in clockwise order) the $k$-cycle $(m_1,\dots, m_k)$ and to a polygon $P=-P=\{m_1, \dots, m_k\}\overset{\cdot}{\cup}\{ i_1, -i_1\}$ (where the $m_i$'s are read in clockwise order) the product of cycles $(i_1, -i_1)(m_1, \dots, m_k)$. 

Examples are given in Figure~\ref{fig:d_n} below and we refer to \cite[Section 5.2]{BG} or \cite{AR} for more details.

\begin{center}
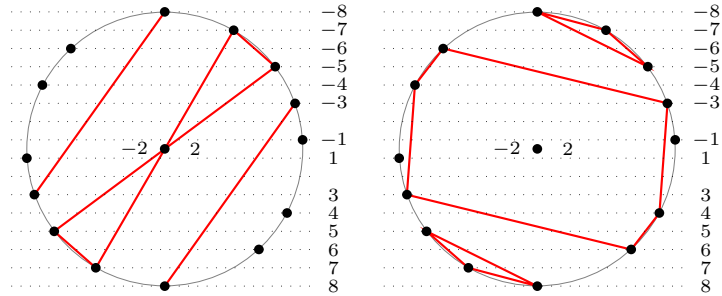
\begin{figure}[h!]

\begin{tabular}{cc}

\psscalebox{1.01}{\begin{pspicture}(-2,-1.92)(2.5,1.92)
\pscircle[linecolor=gray, linewidth=0.2pt](0,0){1.8}

\psline[linecolor=red](0.897,1.56)(1.44,1.08)(0,0)(0.897,1.56)
\psline[linecolor=red](-0.897,-1.56)(-1.44,-1.08)(0,0)(-0.897,-1.56)
\psline[linecolor=red](0,1.8)(-1.697,-0.6)
\psline[linecolor=red](0,-1.8)(1.697,0.6)

\psdots(0,1.8)(0.897,1.56)(-1.2237,1.32)(1.44,1.08)(-1.59197,0.84)(1.697,0.6)(0,0)(1.796,0.12)(-1.796,-0.12)(-1.697,-0.6)(1.59197,-0.84)(-1.44,-1.08)(1.2237,-1.32)(-0.897,-1.56)(0,-1.8)


\psline[linestyle=dotted, linewidth=0.4pt](-2,1.8)(2,1.8)
\psline[linestyle=dotted, linewidth=0.4pt](-2,1.56)(2,1.56)
\psline[linestyle=dotted, linewidth=0.4pt](-2,1.32)(2,1.32)
\psline[linestyle=dotted, linewidth=0.4pt](-2,1.08)(2,1.08)
\psline[linestyle=dotted, linewidth=0.4pt](-2,0.84)(2,0.84)
\psline[linestyle=dotted, linewidth=0.4pt](-2,0.6)(2,0.6)
\psline[linestyle=dotted, linewidth=0.4pt](-2,0.36)(2,0.36)
\psline[linestyle=dotted, linewidth=0.4pt](-2,0.12)(2,0.12)
\psline[linestyle=dotted, linewidth=0.4pt](-2,-1.8)(2,-1.8)
\psline[linestyle=dotted, linewidth=0.4pt](-2,-1.56)(2,-1.56)
\psline[linestyle=dotted, linewidth=0.4pt](-2,-1.32)(2,-1.32)
\psline[linestyle=dotted, linewidth=0.4pt](-2,-1.08)(2,-1.08)
\psline[linestyle=dotted, linewidth=0.4pt](-2,-0.84)(2,-0.84)
\psline[linestyle=dotted, linewidth=0.4pt](-2,-0.6)(2,-0.6)
\psline[linestyle=dotted, linewidth=0.4pt](-2,-0.36)(2,-0.36)
\psline[linestyle=dotted, linewidth=0.4pt](-2,-0.12)(2,-0.12)

\rput(2.2, 1.8){\tiny $-8$}
\rput(2.2, 1.56){\tiny $-7$}
\rput(2.2, 1.32){\tiny $-6$}
\rput(2.2, 1.08){\tiny $-5$}
\rput(2.2, 0.84){\tiny $-4$}
\rput(2.2, 0.6){\tiny $-3$}
\rput(-0.4,0){\tiny $-2$}
\rput(0.4,0){\tiny $2$}
\rput(2.2, 0.12){\tiny $-1$}
\rput(2.2, -0.12){\tiny $1$}
\rput(2.2, -0.6){\tiny $3$}
\rput(2.2, -0.84){\tiny $4$}
\rput(2.2, -1.08){\tiny $5$}
\rput(2.2, -1.32){\tiny $6$}
\rput(2.2, -1.56){\tiny $7$}
\rput(2.2, -1.8){\tiny $8$}
\end{pspicture}}

&

\psscalebox{1.01}{\begin{pspicture}(-2,-1.92)(2.5,1.92)
\pscircle[linecolor=gray, linewidth=0.2pt](0,0){1.8}

\psline[linecolor=red](-1.2237,1.32)(1.697,0.6)(1.59197,-0.84)(1.2237,-1.32)(-1.697,-0.6)(-1.59197,0.84)(-1.2237,1.32)
\psline[linecolor=red](0,1.8)(0.897,1.56)(1.44,1.08)(0,1.8)
\psline[linecolor=red](0,-1.8)(-0.897,-1.56)(-1.44,-1.08)(0,-1.8)

\psdots(0,1.8)(0.897,1.56)(-1.2237,1.32)(1.44,1.08)(-1.59197,0.84)(1.697,0.6)(0,0)(1.796,0.12)(-1.796,-0.12)(-1.697,-0.6)(1.59197,-0.84)(-1.44,-1.08)(1.2237,-1.32)(-0.897,-1.56)(0,-1.8)


\psline[linestyle=dotted, linewidth=0.4pt](-2,1.8)(2,1.8)
\psline[linestyle=dotted, linewidth=0.4pt](-2,1.56)(2,1.56)
\psline[linestyle=dotted, linewidth=0.4pt](-2,1.32)(2,1.32)
\psline[linestyle=dotted, linewidth=0.4pt](-2,1.08)(2,1.08)
\psline[linestyle=dotted, linewidth=0.4pt](-2,0.84)(2,0.84)
\psline[linestyle=dotted, linewidth=0.4pt](-2,0.6)(2,0.6)
\psline[linestyle=dotted, linewidth=0.4pt](-2,0.36)(2,0.36)
\psline[linestyle=dotted, linewidth=0.4pt](-2,0.12)(2,0.12)
\psline[linestyle=dotted, linewidth=0.4pt](-2,-1.8)(2,-1.8)
\psline[linestyle=dotted, linewidth=0.4pt](-2,-1.56)(2,-1.56)
\psline[linestyle=dotted, linewidth=0.4pt](-2,-1.32)(2,-1.32)
\psline[linestyle=dotted, linewidth=0.4pt](-2,-1.08)(2,-1.08)
\psline[linestyle=dotted, linewidth=0.4pt](-2,-0.84)(2,-0.84)
\psline[linestyle=dotted, linewidth=0.4pt](-2,-0.6)(2,-0.6)
\psline[linestyle=dotted, linewidth=0.4pt](-2,-0.36)(2,-0.36)
\psline[linestyle=dotted, linewidth=0.4pt](-2,-0.12)(2,-0.12)

\rput(2.2, 1.8){\tiny $-8$}
\rput(2.2, 1.56){\tiny $-7$}
\rput(2.2, 1.32){\tiny $-6$}
\rput(2.2, 1.08){\tiny $-5$}
\rput(2.2, 0.84){\tiny $-4$}
\rput(2.2, 0.6){\tiny $-3$}
\rput(-0.4,0){\tiny $-2$}
\rput(0.4,0){\tiny $2$}
\rput(2.2, 0.12){\tiny $-1$}
\rput(2.2, -0.12){\tiny $1$}
\rput(2.2, -0.6){\tiny $3$}
\rput(2.2, -0.84){\tiny $4$}
\rput(2.2, -1.08){\tiny $5$}
\rput(2.2, -1.32){\tiny $6$}
\rput(2.2, -1.56){\tiny $7$}
\rput(2.2, -1.8){\tiny $8$}
\end{pspicture}}

\end{tabular}

\caption{Examples of noncrossing partitions $x_1=((3,-8))((7,5,-2))$, $x_2=((8,7,5))[4,6,3][2]$ of type $D_n$.}
\label{fig:d_n}

\end{figure}

\end{center}

For our purposes, we need to slightly modify the graphical representation introduced above, as done in ~\cite{BG}. We proceed as follows: starting from a diagram as above, we then split the middle point into two points with respective labels $i_1$ and $-i_1$, which we move on the vertical axis of the circle in such a way that the condition on the height of the points stays satisfied (see Figure~\ref{fig:d_n_2}). Note that the polygons might be replaced by curvilinear polygons to keep the noncrossing property, and that in general there is not a unique way to split $P$ and $-P$ in the case where $P\neq -P$ and $\pm i_1\in P$ (see Figure~\ref{fig:d_n_nonunique}). In all the cases, we require that each curve corresponding to an edge of a curvilinear polygon is either strictly increasing or decreasing.

\begin{center}

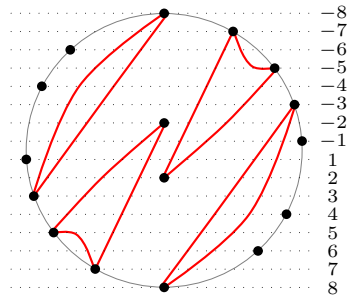
\begin{figure}[h!]

\psscalebox{1.01}{\begin{pspicture}(-2,-1.92)(2.5,1.8)
\pscircle[linecolor=gray, linewidth=0.2pt](0,0){1.8}

\psline[linecolor=red](-0.897,-1.56)(0,0.32)
\pscurve[linecolor=red](-1.44,-1.04)(-0.8,-0.35)(0,0.36)
\pscurve[linecolor=red](-0.92,-1.53)(-1.15,-1.1)(-1.44,-1.08)

\psline[linecolor=red](0.897,1.56)(0,-0.32)
\pscurve[linecolor=red](1.44,1.04)(0.8,0.35)(0,-0.36)
\pscurve[linecolor=red](0.92,1.53)(1.15,1.1)(1.44,1.08)

\pscurve[linecolor=red](0,1.8)(-1.1,0.85)(-1.697,-0.56)
\pscurve[linecolor=red](0.02,1.76)(-0.65,0.85)(-1.697,-0.6)

\pscurve[linecolor=red](0,-1.8)(1.1,-0.85)(1.697,0.56)
\pscurve[linecolor=red](-0.02,-1.76)(0.65,-0.85)(1.697,0.6)

\psdots(0,1.8)(0.897,1.56)(-1.2237,1.32)(1.44,1.08)(-1.59197,0.84)(1.697,0.6)(0,0.36)(1.796,0.12)(-1.796,-0.12)(0,-0.36)(-1.697,-0.6)(1.59197,-0.84)(-1.44,-1.08)(1.2237,-1.32)(-0.897,-1.56)(0,-1.8)


\psline[linestyle=dotted, linewidth=0.4pt](-2,1.8)(2,1.8)
\psline[linestyle=dotted, linewidth=0.4pt](-2,1.56)(2,1.56)
\psline[linestyle=dotted, linewidth=0.4pt](-2,1.32)(2,1.32)
\psline[linestyle=dotted, linewidth=0.4pt](-2,1.08)(2,1.08)
\psline[linestyle=dotted, linewidth=0.4pt](-2,0.84)(2,0.84)
\psline[linestyle=dotted, linewidth=0.4pt](-2,0.6)(2,0.6)
\psline[linestyle=dotted, linewidth=0.4pt](-2,0.36)(2,0.36)
\psline[linestyle=dotted, linewidth=0.4pt](-2,0.12)(2,0.12)
\psline[linestyle=dotted, linewidth=0.4pt](-2,-1.8)(2,-1.8)
\psline[linestyle=dotted, linewidth=0.4pt](-2,-1.56)(2,-1.56)
\psline[linestyle=dotted, linewidth=0.4pt](-2,-1.32)(2,-1.32)
\psline[linestyle=dotted, linewidth=0.4pt](-2,-1.08)(2,-1.08)
\psline[linestyle=dotted, linewidth=0.4pt](-2,-0.84)(2,-0.84)
\psline[linestyle=dotted, linewidth=0.4pt](-2,-0.6)(2,-0.6)
\psline[linestyle=dotted, linewidth=0.4pt](-2,-0.36)(2,-0.36)
\psline[linestyle=dotted, linewidth=0.4pt](-2,-0.12)(2,-0.12)

\rput(2.2, 1.8){\tiny $-8$}
\rput(2.2, 1.56){\tiny $-7$}
\rput(2.2, 1.32){\tiny $-6$}
\rput(2.2, 1.08){\tiny $-5$}
\rput(2.2, 0.84){\tiny $-4$}
\rput(2.2, 0.6){\tiny $-3$}
\rput(2.2, 0.36){\tiny $-2$}
\rput(2.2, 0.12){\tiny $-1$}
\rput(2.2, -0.12){\tiny $1$}
\rput(2.2, -0.36){\tiny $2$}
\rput(2.2, -0.6){\tiny $3$}
\rput(2.2, -0.84){\tiny $4$}
\rput(2.2, -1.08){\tiny $5$}
\rput(2.2, -1.32){\tiny $6$}
\rput(2.2, -1.56){\tiny $7$}
\rput(2.2, -1.8){\tiny $8$}
\end{pspicture}}

\caption{Splitting of the middle point. Here the noncrossing partition is $x_1=((3,-8))((7,5,-2))$ as in the left picture of Figure~\ref{fig:d_n}.}
\label{fig:d_n_2}

\end{figure}

\end{center}

\subsubsection{Explicit description of the map from noncrossing partitions to $c$-sortable elements}

As in type $A_n$, we introduce a total order $\prec$ on the set $\Pol(x)$ of polygons of $x\in\NC(W_{D_n}, c)$, in such a way that $P\prec Q$ if and only if $-Q\prec -P$. We begin by defining an order $<$ on the set $$\{ Q~|~ Q\in\Pol(x), ~Q=-Q\} \cup\{ P\cup-P~|~P\in\Pol(x), P\neq -P\}.$$
                           
Before doing it, we define a property of a pair $P\neq -P$ of polygons of $x$ (and the restriction of $\prec$ to such a pair, which will be useful for the definition of $\prec$ on the whole set $\Pol(x)$):

\begin{definition}[Special polygons]

Let $P\in\Pol(x)$ with $|P|\geq 2$ and $P\neq -P$. We say that the pair $\{P,-P\}$ is \defn{special} if
\begin{enumerate}
\item One of the two polygons has exactly one index in $\{1, 2\}$ with all other indices in $\{-1, \dots, -n\}$. Note that it can happen that both $P$ and $-P$ satisfy this condition,
\item There is no polygon lying between $P$ and $-P$ (meaning that if a horizontal line crosses both $P$ and $-P$, it crosses no polygon in between) and the polygon among $P$ and $-P$ having the smallest index, say $P$, has at most one vertex of $-P$ which lies strictly at its right. 
\end{enumerate}
In this situation we set $P\prec -P$.
\end{definition}

Note that there can be at most one pair of special polygons.

\begin{exple}
Let $c=(-4, -3, -1, 4, 3, 1)(2,-2)$, $x=((-4,-3,2))$. Writing $P=\{-4, -3,2\}$ we have that $\{P, -P\}$ is special with $P\prec -P$. Note that there is not a unique way of graphically representing $x$ in that case (see Figure~\ref{fig:d_n_nonunique}). Depending on how we represent the noncrossing partition, $P$ is stricly at the right of $-P$, or the opposite. Another example is given by $x=((-4,-2,1))$. In this last example, $P:=\{-4,-2,1\}$ appears strictly at the left of $-P$ in the noncrossing representation but we nevertheless have $P\prec -P$. This is one of the reason why we have to introduce special polygons: in situations where $P\neq -P$ but the pair is not special, in the case where $P$ appears strictly at the right of $-P$, we will have $P\prec -P$ (which is not the case in the last situation above with $x=((-4,-2,1))$, where $P\prec -P$ but $-P$ appears strictly at the right of $P$). 
\end{exple}

\begin{center}

\begin{figure}[h!]

\begin{tabular}{cc}

\psscalebox{1.01}{\begin{pspicture}(-2,-1.92)(2.5,1.8)
\pscircle[linecolor=gray, linewidth=0.2pt](0,0){1.8}

\pscurve[linecolor=red](0,1.8)(0.6, 1.5)(1.2237,1.32)
\psline[linecolor=red](1.2237,1.32)(0,-0.84)
\pscurve[linecolor=red](0,-0.84)(0.3,0.8)(0,1.8)

\pscurve[linecolor=red](0,-1.8)(-0.6, -1.5)(-1.2237,-1.32)
\psline[linecolor=red](-1.2237,-1.32)(0,0.84)
\pscurve[linecolor=red](0,0.84)(-0.3,-0.8)(0,-1.8)




\psdots(0,1.8)(1.2237,1.32)(0,0.84)(1.74,0.36)(-1.74,-0.36)(0,-0.84)(-1.2237,-1.32)(0,-1.8)


\psline[linestyle=dotted, linewidth=0.4pt](-2,1.8)(2,1.8)
\psline[linestyle=dotted, linewidth=0.4pt](-2,1.32)(2,1.32)
\psline[linestyle=dotted, linewidth=0.4pt](-2,0.84)(2,0.84)
\psline[linestyle=dotted, linewidth=0.4pt](-2,0.36)(2,0.36)
\psline[linestyle=dotted, linewidth=0.4pt](-2,-1.8)(2,-1.8)
\psline[linestyle=dotted, linewidth=0.4pt](-2,-1.32)(2,-1.32)
\psline[linestyle=dotted, linewidth=0.4pt](-2,-0.84)(2,-0.84)
\psline[linestyle=dotted, linewidth=0.4pt](-2,-0.36)(2,-0.36)

\rput(2.2, 1.8){\tiny $-4$}
\rput(2.2, 1.32){\tiny $-3$}
\rput(2.2, 0.84){\tiny $-2$}
\rput(2.2, 0.36){\tiny $-1$}
\rput(2.2, -0.36){\tiny $1$}
\rput(2.2, -0.84){\tiny $2$}
\rput(2.2, -1.32){\tiny $3$}
\rput(2.2, -1.8){\tiny $4$}
\end{pspicture}} &

\psscalebox{1.01}{\begin{pspicture}(-2,-1.92)(2.5,1.8)
\pscircle[linecolor=gray, linewidth=0.2pt](0,0){1.8}

\pscurve[linecolor=red](0,1.8)(0.6, 1.5)(1.2237,1.32)
\pscurve[linecolor=red](1.2237,1.32)(-0.28,0.8)(0,-0.84)
\pscurve[linecolor=red](0,-0.84)(-0.5,0.8)(0,1.8)

\pscurve[linecolor=red](0,-1.8)(-0.6, -1.5)(-1.2237,-1.32)
\pscurve[linecolor=red](-1.2237,-1.32)(0.28,-0.8)(0,0.84)
\pscurve[linecolor=red](0,0.84)(0.5,-0.8)(0,-1.8)



\psdots(0,1.8)(1.2237,1.32)(0,0.84)(1.74,0.36)(-1.74,-0.36)(0,-0.84)(-1.2237,-1.32)(0,-1.8)


\psline[linestyle=dotted, linewidth=0.4pt](-2,1.8)(2,1.8)
\psline[linestyle=dotted, linewidth=0.4pt](-2,1.32)(2,1.32)
\psline[linestyle=dotted, linewidth=0.4pt](-2,0.84)(2,0.84)
\psline[linestyle=dotted, linewidth=0.4pt](-2,0.36)(2,0.36)
\psline[linestyle=dotted, linewidth=0.4pt](-2,-1.8)(2,-1.8)
\psline[linestyle=dotted, linewidth=0.4pt](-2,-1.32)(2,-1.32)
\psline[linestyle=dotted, linewidth=0.4pt](-2,-0.84)(2,-0.84)
\psline[linestyle=dotted, linewidth=0.4pt](-2,-0.36)(2,-0.36)

\rput(2.2, 1.8){\tiny $-4$}
\rput(2.2, 1.32){\tiny $-3$}
\rput(2.2, 0.84){\tiny $-2$}
\rput(2.2, 0.36){\tiny $-1$}
\rput(2.2, -0.36){\tiny $1$}
\rput(2.2, -0.84){\tiny $2$}
\rput(2.2, -1.32){\tiny $3$}
\rput(2.2, -1.8){\tiny $4$}
\end{pspicture}}

\end{tabular}

\caption{Non unique graphical representation of the noncrossing partition $x=((-4,-3,2))$.}
\label{fig:d_n_nonunique} 

\end{figure}
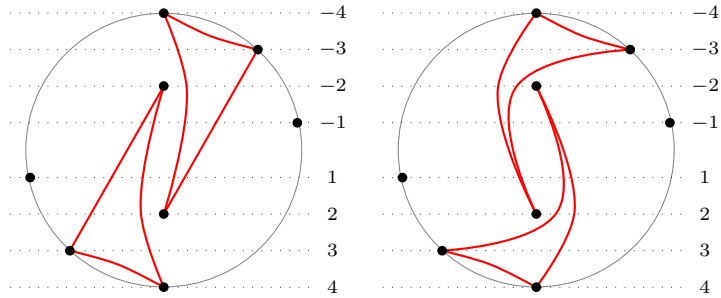

\end{center}


The order $<$ is defined as follows. We apply an inductive procedure similar to the one defined in type $A_n$ when we defined $\prec$, removing polygons or pairs of polygons at each step. We distinguish three situations (note that we define relative positions of polygons exactly as in type $A_n$, see Definition~\ref{def_rel}): 
\begin{enumerate}
\item Assume that $x$ has no pair of special polygons. In the graphical representation of $x$, let $P$ be the unique polygon such that $P$ has no polygon lying strictly at its right and $\max (P)$ is minimal among the polygons satisfying this property. Then $P\cup -P$ appears first in the total order, and if $P\neq -P$ we set $P\prec -P$. We then remove $P \cup -P$ and all its vertices from the graphical representation. 
\item If there is a pair of special polygons $\{P, -P\}$ with $P\prec -P$ and $P$ has a polygon $Q\neq -P$ strictly at its right, then we do as in the first case. Note that since there is at most one pair of special polygons, the pair which will be removed in this case is not special since there is at least one polygon which lies strictly at the right of $P$.  
\item If there is a special pair of polygons $\{P, -P\}$ with $P\prec -P$ and there is no polygon $Q$ distinct from $-P$ lying strictly at the right of $P$, then consider the unique polygon $Q$ in $\Pol(x)\backslash\{-P\}$ having no polygon lying strictly at its right and such that $\max(Q)$ minimal (note that we can have $Q$=$P$). If $\max(Q) < \max(P)$ then $Q\cup -Q$ appears first in the total order and we set $Q\prec -Q$ (note that we cannot have $Q=-Q$), and we remove them as well as their vertices from the graphical representation. Otherwise $P\cup -P$ appears first in the total order, and we remove it as well. 
\end{enumerate}

We then go on inductively. Note that a polygon $Q$ with $Q=-Q$ is necessarily the top element for $<$. 

The order $\prec$ is then defined as follows. If there is no symmetric polygon in $\Pol(x)$, label the polygons of $x$ such that $$(P_1 \prec -P_1) < (P_2\prec -P_2) < \dots < (P_{k} \prec -P_{k}).$$ Then set $$P_1\prec P_2 \prec \dots \prec P_k \prec -P_k \prec \dots \prec -P_2 \prec - P_1.$$
If there is $Q\in\Pol(x)$ which is symmetric, then as noticed above $Q$ has to be the top element for $<$. Labeling the polygons of $x$ such that  $$(P_1 \prec -P_1) < (P_2\prec -P_2) < \dots < (P_{k} \prec -P_{k}) <Q,$$ we then set $$P_1\prec P_2 \prec \dots \prec P_k \prec Q \prec -P_k \prec \dots \prec -P_2 \prec - P_1.$$   

\begin{exple}
Let $x_2=((8,7,5))[4,6,3][2]$ be as in Figure~\ref{fig:d_n}. We have $$(\{-8,-7,-5\} \prec \{8,7,5\}) < (\{-1\} \prec \{1\}) < \{\pm 2, \pm 3, \pm 4, \pm 6 \},$$
Hence the order $\prec$ is given by $$\{-8,-7,-5\} \prec \{-1\} \prec \{ \pm 2, \pm 3, \pm 4, \pm 6\} \prec \{1\} \prec \{8,7,5\}.$$
\end{exple}

We use the order $\prec$ to define a permutation $S_c(x)\in W_{D_n}$ in a similar way as we did in type $A_n$. We first define a permutation $S_c(x)'$ as follows: if $P$ is first in the total order $\prec$, then we order the vertices of $P$ as $\{ d_1 > d_2 > \dots > d_k > -d_k > \dots > -d_2 > -d_1\}$ and define $S_c(y)$ on the subset $\{ -n, \dots, -n+k-1\}$ by $-n \mapsto d_1, \dots, -n+k-1\mapsto d_k$. We then consider the next $Q$ appearing in $\prec$, and define $S_c(x)$ in the same way on the subset $\{-n+k, \dots, -n+k+\ell-1\}$, and so on inductively.

Note that since $P\prec Q$ if and only if $-Q \prec - P$, the permutation $S_c(x)'$ is signed, but in some cases it might not lie in $W_{D_n}$. If $S_c(x)'\in W_{D_n}$ then we set $S_c(x):=S_c(x)'$. If $S_c(x)'$ does not lie in $W_{D_n}$, then in the line notation $\pi_{-n} \cdots \pi_{-1} \pi_1 \cdots \pi_n$ of the permutation $S_c(x)'$, we swap the two centermost entries, that is, the entries $\pi_1$ and $\pi_{-1}$ to get a new line notation and define $S_c(x)\in W_{D_n}$ as the unique permutation having this line notation. 

\begin{exple}\label{ex_dn}
Let $x_1=((3, -8))((7,5,-2))$ be as in Figures~\ref{fig:d_n} and \ref{fig:d_n_2}. The total order $<$ is given by $$\pm\{1\} < \pm\{4\} < \pm \{6\} < \pm\{8, -3\} < \pm\{2,-5,-7\}.$$
The first half of the order $\prec$ is given by $$\{-1\}\prec\{4\}\prec\{6\}\prec\{8,-3\}\prec\{-7,-5,2\}.$$
We have $\{-1\}\prec \{1\}$, hence we set $S_c(x_1)'(-n)=-1$, (and therefore $S_c(x_1)'(n)=1$). Going on inductively, we get (in the line notation) a permutation $$S_c(x_1)'=(-1) 4 6 8 (-3) 2 (-5) (-7) 7 5 (-2) 3 (-8) (-6) (-4) 1.$$  We have that $|\{ i\in \{1,\dots, n\}~|~S_c(x_1)'(i)<0\}|=4$, hence the above permutation is in $W_{D_n}$ and we set $S_c(x_1)=S_c(x_1)'$. 

Using the characterization of $c$-sortable elements in type $D_n$ given by \cite[Lemma 4.11]{Reading}, we see that $S_c(y)$ is $c$-sortable. Moreover its cover reflections are given by $((5,7))$, $((-2,5))$, $((3, -8))$ (they are obtained as follows: if we denote the line notation of the permutation by $\pi_{-n} \cdots \pi_{-1} \pi_1 \cdots \pi_n$, then the cover reflections are precisely those $((\pi_i, \pi_{i+1}))$ such that $\pi_{i+1} >  \pi_i$, where $i=1,\dots, n-1$, and $((\pi_{-1}, \pi_2))$ in the case where $\pi_2 < \pi_{-1}$), which implies that $\Read(S_c(x_1))=x_1$. This is a general fact, proven in Theorem~\ref{thm_d} below. 
\end{exple}

\begin{rmq}\label{rmq:initial_d}
The following can be checked by a direct computation. If $s=s_i$, $i\geq 2$, then $s$ is initial in $c$ if and only if $i\in L_c$, $i+1\in R_c$. If $i_1=2$, then $s_0$ is initial in $c$ if and only if $-1\in L_c$, while $s_1$ is initial if and only if $1\in L_c$. If $i_1=1$, then $s_0$ is initial in $c$ if and only if $s_1$ is initial in $c$, which happens if and only if $-2$ is in $L_c$.    
\end{rmq}

\begin{theorem}\label{thm_d}
Let $x\in\NC(W_{D_n}, c)$. The element $S_c(x)\in W_{D_n}$ is $c$-sortable and we have $\Read(S_c(x))=x$. That is, we have $S_c=\Read^{-1}$.
\end{theorem}

\begin{proof}
Let $x\in \NC(W_{D_n}, c)$ and let $y=\Read^{-1}(x)$. We show by Cambrian recurrence that $S_c(\Read(y))=y$, which concludes: note that here the induction is on the rank of a Coxeter group of type $D_n$ and the length of Coxeter sortable elements inside it, i.e., if $y$ is $c$-sortable in $W$ we will assume that $S_{c'}(\Read(y'))=y'$ for every $c'$-sortable element $y'\in W$ with $\ell(y') < \ell(y)$ and for every  sortable element $y'$ in a standard parabolic subgroup $W'$ of $W$ of type $D_k$, $k<n$ (with $D_1=A_1, D_2=A_1\times A_1, D_3=A_3$). If $y=1$ then the claim is trivial. If $y=s\in S$, then $x$ has two polygons $P_1=\{i, i+1\}$, $P_2=\{-i, -i-1\}$ with $P_2\prec P_1$ in the case where $s=s_i$ for some $i\geq 1$. We have $\Read(y)=y$ and it follows from the definition of $S_c(y)$ that $S_c(y)=y$, which shows the claim. If $y=s_0$, then we have two polygons $P_1=(-1, 2)$, $P_2=(1, -2)$. The pair is special and we have $P_2\prec P_1$. The permutation which we obtain as $S_c(y)'$ is $(-2, 1, 2, -1)$, but this is not in $W_{D_n}$ hence to get $S_c(y)$ we permute the two centermost entries in the line notation of $S_c(y)'$. This yields $S_c(y)=y=s_0$. If $W$ has type $A_1\times A_1$ (here we use the same graphical model with one simple reflection being $s_0$ and the other one being $s_1$), then it is again easy to check the claim directly. 

Now assume that $\ell(y)>0$. Let $s$ be initial in $c$.\\ 

\noindent\textbf{First case: $sy > y$}. By Lemma~\ref{lem:read1}, it implies that $y$ and $\Read(y)\in W_{<s>}$ and that $y$ is $sc$-sortable. We have $\Read_c(y)=\Read_{sc}(y)$. We want to show that $S_c(\Read(y))=y$. If $W_{<s>}$ is of type $D_{n-1}$, then $s=s_{n-1}$ and by induction we have that $S_{sc}(\Read_{sc}(y))=y$. Now $\Read_{sc}(y)=\Read_c(y)$ and this noncrossing partitions fixes $-n$. It follows that $S_{c}(\Read_c(y))$ is obtained from $S_{sc}(\Read_{sc}(y))$ by extending the permutation by a trivial action on $\{-n, n\}$ which shows the claim. Now assume that $s=s_i$, $i\in\{3, \dots, n-1\}$. It follows that $W_{<s>}$ is a product of a Coxeter group $W_1$ of type $D_k$ for some $k\geq 3$ and a Coxeter group $W_2$ of type $A$. In this situation $y$ decomposes uniquely as a product $y=y_1 y_2$ with $y_i\in W_i$ and each $y_i$ is sortable in the corresponding factor for the restrictions $(sc)_1$ and $(sc)_2$ of $sc$ to these factors. We have that $\Read_c(y)=\Read_{sc}(y)=\Read_{(sc)_1}(y_1)\Read_{(sc)_2}(y_2)$ (with $\Read_{(sc)_i}(y_i)\in W_i$) and it is not difficult to observe by construction of the maps $S_{c'}$ (for various $c'$ here) that $S_{c}(\Read_{(sc)_1}(y_1)\Read_{(sc)_2}(y_2))=S_{(sc)_1}(\Read_{(sc)_1}(y_1)) S_{(sc)_2}(\Read_{(sc)_2}(y_2))$, where $S_{(sc)_2}$ coincides with the map which we defined in type $A$ (here $W_2$ is isomorphic to its restriction to the $W_2$-stable subset $\{i+1, \dots, n\}$, which is isomorphic to the symmetric group $\mathfrak{S}_{n-i}$. We use this identification implicitly). Hence by induction on the rank together with the type $A$ situation we get $S_{c}(\Read(y))= y_1 y _2=y$. The case where $s=s_2$ is similar but with a Coxeter group of type $A_1\times A_1$ insted of a Coxeter group of type $D_n$. In the case where $s=s_1$ (resp. $s=s_0$), the map $S_{sc}$ coincides with the map which we defined on a Coxeter group of type $A_{n-1}$ (in this group the subset $\{1, \dots, n\}$ (resp. $\{-1, 2, \dots, n\}$) is stable by the action of every permutation, the isomorphism with $\mathfrak{S}_{n}$ is immediate).\\ 

\noindent\textbf{Second case: $sy < y$ and $s$ is a cover reflection of $y$}. This precisely means that $s\leq_T \Read(y)$ (by Theorem~\ref{thm_reading}), hence that $s$ corresponds to a symmetric pair of edges or diagonals of either a symmetric polygon $Q$ or a pair $P\neq -P$ of polygons of $x=\Read(y)$ (if $Q=-Q$ and $s=s_0$ or $s_1$, we also call \textit{diagonal} a segment joining a vertex of $Q$ to one of the middle points). The element $sy$ is $scs$-sortable by Lemma~\ref{lem:read2} and by induction on length we have that $S_{scs}(\Read_{scs}(sy))=sy$. Note that that $\Read_{scs}(sy)= \Read(y)s$ by~\cite[Lemma 6.5]{Reading}. Hence to prove the claim it suffices to show that $$s S_{scs}(\Read(y)s)=S_c(\Read(y)).$$
Assume that $s=s_i$, $i\geq 2$. Then $i\in L_c$, $i+1\in R_c$ (see Remark~\ref{rmq:initial_d}). Assume that $(i, i+1)$ is a diagonal of a polygon $P$ with $P\neq -P$ of $x=\Read(y)$. The $scs$-noncrossing partition $\Read(y) s$ is given by splitting $P$ into two polygons $P_1=\{j\in P~|~j> i+1\}\cup\{i\}$ and $P_2=\{j\in P~|~j <i\} \cup\{i+1\}$ (see Figure~\ref{fig:d_n_1}). But the labeling corresponding to $scs$ is obtained from that of $c$ by moving $i$ to $R_c$ and $i+1$ to $L_c$. Hence if we compare the total order on polygons of $\Read(y)$ to that of $\Read(y)s$, in the case where $P\prec -P$ we get that $P$ is replaced by $P_1\prec P_2$ (and $-P$ consistently) while everything else is unchanged: every polygon which is removed before (resp. after) $P$ in $\Read(y)$ stays removed before $P_1$ (resp. after $P_2$) in the same order, while there can be nothing removed in between. Now if $-P\prec P$, then $-P$ is replaced by $(-P_2)\prec(-P_1)$. Hence in both cases, $P$ is replaced by $P_1\prec P_2$. Moreover, if $P=\{ d_1 > ... > d_\ell > i+1 > i > d_{\ell+1} >\dots > d_m\}$, then the line notations of $S_c(\Read(y))'$ and of $S_{scs}(\Read(y)s)'$ coincide except in the entries corresponding to $P$ (and $-P$), where in $S_c(\Read(y))'$ we have $$d_1 \cdots d_\ell (i+1) i d_{\ell+1} \cdots d_m$$ while for $S_{scs}(\Read(y)s)'$ we have  $$d_1 \cdots d_\ell i (i+1) d_{\ell+1} \cdots d_m$$ (and consistently for $-P$). This shows that $s S_{scs}(\Read(y)s)'=S_c(\Read(y))'$, which implies that $s S_{scs}(\Read(y)s)=S_c(\Read(y))$. Now if $(i, i+1)$ is a diagonal of $Q=-Q$, the proof is similar and left to the reader (note that in that case $Q$ is central in the total order $\prec$). 

Now assume that $s=s_0$ and assume in addition that $i_1=2$ (we then have $-1\in L_c$ by Remark~\ref{rmq:initial_d}). If $s$ is a diagonal of a symmetric pair of polygons $P\neq -P$, then $P$ is split into two polygons $P_2$, $P_1$ as in the above case and one checks that we get the same as in the case $s=s_i, i\geq 2$: $P$ is replaced by $P_1 \prec P_2$ while everything stays unchanged everywhere else (in that case passing from $c$ to $scs$ moves $\pm 1$ to the middle while $2$ is moved to $L_c$). Assume that $s$ corresponds to a pair of diagonals of $Q=-Q$. Note that $Q$ is central in the total order $\prec$ on $\Pol(\Read(y))$. Writing the permutation corresponding to $Q$ as $(-j_1, \dots, -j_\ell, j_1, \dots, j_\ell)(2, -2)$ with $j_i\geq 1$, $j_1=1$, $j_i < j_{i+1}$, in the graphical representation of $\Read(y) s$ we have that the polygon $Q$ is replaced by a pair of symmetric polygons $P, -P$ with $P=\{j_2, \dots, j_\ell, -j_1=-1, -2\}$. Now in the labeling corresponding to $scs$, the indices $1$ and $-2$ get moved as described above which yields $P\prec -P$ ($P$ is strictly at the right of $-P$ and the pair cannot be special). Moreover, in the total order $\prec$ on the polygons, nothing gets changed for the other polygons which $\Read(y)$ and $\Read(y)s$ have in commmon. Hence if we compare $S_c(\Read(y))'$ with $S_{scs}(\Read(y)s)'$, there is only a difference in the entries corresponding to $Q$ / $P\cup -P$, in the $4$ centermost entries which in the first case are $2 1 (-1) (-2)$ while in the second case they are $(-1) (-2) 2 1$. This shows that $s S_{scs}(\Read(y)s)'=S_c(\Read(y))'$, which implies that $s S_{scs}(\Read(y)s)=S_c(\Read(y))$, which is precisely what we wanted to show. The case $s=s_1$, as well as the cases where $i_1=1$ are similar and left to the reader.\\ 

\noindent\textbf{Third case: $sy < y$ and $s$ is not a cover reflection of $y$}. It means that $s\notin\leqt \Read(y)$, but $y, \Read(y)\notin W_{<s>}$. Assume that $s=s_i$, $i\geq 2$. In that case by~\cite[Lemma 6.5]{Reading} we have that $\Read_{scs}(sy)= s \Read(y) s$. Let $s=s_i$, $i\geq 2$. We have $i\in L_c$, $i+1\in R_c$ by Remark~\ref{rmq:initial_d}. If $i\in P\in\Pol(\Read(y))$, then $i+1\in Q\in Pol(\Read(y))$ with $Q\neq P$ (otherwise we would have $s\leq_T\Read(y)$). Note that the fact that $\Read(y)\notin W_{<s>}$ forces at least one polygon to have an index $\leq i$ and another index $\geq i+1$. When passing from $\Read(y)$ to $s\Read(y)s$, the polygon $P$ is replaced by a polygon $\overline{P}$ where the index $i$ is replaced by $i+1$, and $Q$ is replaced by $\overline{Q}$ where $i+1$ is replaced by $i$ (and $-Q, -P$ are replaced by $-\overline{Q}$, $-\overline{P}$). Since the labeling for $scs$ is the same as for $c$ except that $i\in R_{scs}$ and $i+1\in L_{scs}$, it is easy to check (by distinguishing several possible relative configurations of polygons) that the total order $\prec$ on polygons of $\Read(y)$ and $s\Read(y)s$ is the same except that $P, Q$ are replaced by $\overline{P}$, $\overline{Q}$ (The condition that there must be at least one polygon with an index $\leq i$ and another one $\geq i+1$ has to be used here to check this fact, otherwise it can be wrong that the order after overlining $\pm P, \pm Q$ stays the same; note that the cases $P=-P$ or $Q=-Q$ can appear. See Figure~\ref{fig:d_n_autre}, where one of the situations is treated). It follows that the entries corresponding to $S_c(\Read(y))'$ and to $S_{scs}(s\Read(y) s)'$ are the same except that $i$ and $i+1$ are exchanged, hence that $s S_c(\Read(y))= S_{scs}(s \Read(y) s)$, which is what we have to show.

Now assume that $i_1=1$ and $s=s_0$, which implies that $-2\in L_c$ (see Remark~\ref{rmq:initial_d}). Again, the situation is such that $-1\in P$, $2\in Q$ with $Q\neq P$ (otherwise $s\leq_T \Read(y)$). The condition $\Read(y)\notin W_{<s>}$ guarantees that there must be at least one polygon containing both positive and negative indices. Assume that $P\neq -P$; the case where $P=-P$ is similar and left to the reader. Note that $Q\neq -Q$, otherwise $\pm 1\in Q$ (which implies that $P=Q$). In $s\Read(y) s$, we have that $1\in L_c$ while the middle indices are $\pm 2$. To get $s\Read(y)s$ from $\Read(y)$, one replaces the polygon $P$ by a polygon $\overline{P}$ with $2$ instead of $-1$, while $Q$ is replaced by a polygon $\overline{Q}$ with $-1$ instead of $2$ (and $-Q, -P$ are replaced by $-\overline{Q}$, $-\overline{P}$). We have to show that $s S_c(\Read(y))= S_{scs}(s \Read(y) s)$ and we proceed as in the previous case to compare the order $\prec$ on $\Read(y)$ with the order $\prec$ on $s\Read(y)s$. In most of the cases, the order $\prec$ on polygons of $\Read(y)$ after overlining $P$ and $Q$ coincides with the order $\prec$ on polygons of $s\Read(y)s$, but in some situations it may happen that it is not the case; but in these situations, the only difference is that the polygon $P$ is reduced to the index $\{-1\}$, hence it is directly followed in $\prec$ by $-P$, while $\{-2\}=-\overline{P}$ is directly followed in $\prec$ by $\overline{P}$ (see Figure~\ref{fig:d_n_autre_2} for an example). Hence in that case, the two centermost entries in the line notation of either $S_c(\Read(y))$ or $S_{scs}(s\Read(y)s)$ (which correspond to the indices of $\pm P$ and $\pm \overline{P}$ respectively) have to be exchanged to get an element in $W_{D_n}$. We graphically represented two of the various possible situations in Figure~\ref{fig:d_n_autre_2} and \ref{fig:d_n_autre_3} and leave it to the reader to check it in all the other situations since this is elementary geometry, depending on which polygon (among $\pm P$, $\pm Q$ or another one) has both positive and negative indices. It follows that in the line notation for $S_{scs}(s\Read(y) s)'$ and $S_c(\Read(y))'$, we have the same entries (up to permutation of the two centermost entries in the case where one of the two permutations is not in $W_{D_n}$) except that $1$ and $-2$ (and $-1$ and $2$) are exchanged. This shows that $s S_c(\Read(y))= S_{scs}(s \Read(y) s)$.

The case $s=s_1$ as well as the cases $i_1=2$, $s=s_0$ or $s=s_1$ are similar, and left to the reader.

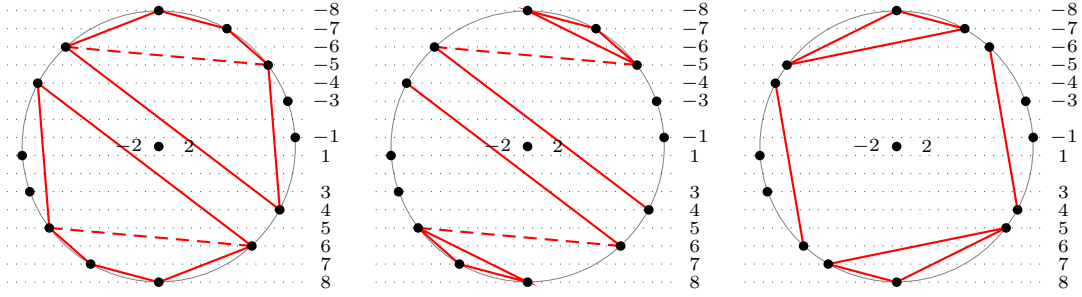
\begin{figure}  

\begin{tabular}{ccc}

\begin{pspicture}(-2,-1.92)(2.5,1.92)
\pscircle[linecolor=gray, linewidth=0.2pt](0,0){1.8}

\psline[linecolor=red](0,-1.8)(1.2237,-1.32)(-1.59197,0.84)(-1.44,-1.08)(-0.897,-1.56)(0,-1.8)

\psline[linecolor=red](0,1.8)(-1.2237,1.32)(1.59197,-0.84)(1.44,1.08)(0.897,1.56)(0,1.8)

\psline[linecolor=red, linestyle=dashed](-1.44,-1.08)(1.2237,-1.32)
\psline[linecolor=red, linestyle=dashed](1.44,1.08)(-1.2237,1.32)

\psdots(0,1.8)(0.897,1.56)(-1.2237,1.32)(1.44,1.08)(-1.59197,0.84)(1.697,0.6)(0,0)(1.796,0.12)(-1.796,-0.12)(-1.697,-0.6)(1.59197,-0.84)(-1.44,-1.08)(1.2237,-1.32)(-0.897,-1.56)(0,-1.8)


\psline[linestyle=dotted, linewidth=0.4pt](-2,1.8)(2,1.8)
\psline[linestyle=dotted, linewidth=0.4pt](-2,1.56)(2,1.56)
\psline[linestyle=dotted, linewidth=0.4pt](-2,1.32)(2,1.32)
\psline[linestyle=dotted, linewidth=0.4pt](-2,1.08)(2,1.08)
\psline[linestyle=dotted, linewidth=0.4pt](-2,0.84)(2,0.84)
\psline[linestyle=dotted, linewidth=0.4pt](-2,0.6)(2,0.6)
\psline[linestyle=dotted, linewidth=0.4pt](-2,0.36)(2,0.36)
\psline[linestyle=dotted, linewidth=0.4pt](-2,0.12)(2,0.12)
\psline[linestyle=dotted, linewidth=0.4pt](-2,-1.8)(2,-1.8)
\psline[linestyle=dotted, linewidth=0.4pt](-2,-1.56)(2,-1.56)
\psline[linestyle=dotted, linewidth=0.4pt](-2,-1.32)(2,-1.32)
\psline[linestyle=dotted, linewidth=0.4pt](-2,-1.08)(2,-1.08)
\psline[linestyle=dotted, linewidth=0.4pt](-2,-0.84)(2,-0.84)
\psline[linestyle=dotted, linewidth=0.4pt](-2,-0.6)(2,-0.6)
\psline[linestyle=dotted, linewidth=0.4pt](-2,-0.36)(2,-0.36)
\psline[linestyle=dotted, linewidth=0.4pt](-2,-0.12)(2,-0.12)

\rput(2.2, 1.8){\tiny $-8$}
\rput(2.2, 1.56){\tiny $-7$}
\rput(2.2, 1.32){\tiny $-6$}
\rput(2.2, 1.08){\tiny $-5$}
\rput(2.2, 0.84){\tiny $-4$}
\rput(2.2, 0.6){\tiny $-3$}
\rput(-0.4,0){\tiny $-2$}
\rput(0.4,0){\tiny $2$}
\rput(2.2, 0.12){\tiny $-1$}
\rput(2.2, -0.12){\tiny $1$}
\rput(2.2, -0.6){\tiny $3$}
\rput(2.2, -0.84){\tiny $4$}
\rput(2.2, -1.08){\tiny $5$}
\rput(2.2, -1.32){\tiny $6$}
\rput(2.2, -1.56){\tiny $7$}
\rput(2.2, -1.8){\tiny $8$}
\end{pspicture}
&

\begin{pspicture}(-2,-1.92)(2.5,1.92)
\pscircle[linecolor=gray, linewidth=0.2pt](0,0){1.8}

\psline[linecolor=red](1.2237,-1.32)(-1.59197,0.84)

\psline[linecolor=red](-1.44,-1.08)(-0.897,-1.56)(0,-1.8)(-1.44,-1.08)

\psline[linecolor=red](-1.2237,1.32)(1.59197,-0.84)

\psline[linecolor=red](1.44,1.08)(0.897,1.56)(0,1.8)(1.44,1.08)

\psline[linecolor=red, linestyle=dashed](-1.44,-1.08)(1.2237,-1.32)
\psline[linecolor=red, linestyle=dashed](1.44,1.08)(-1.2237,1.32)

\psdots(0,1.8)(0.897,1.56)(-1.2237,1.32)(1.44,1.08)(-1.59197,0.84)(1.697,0.6)(0,0)(1.796,0.12)(-1.796,-0.12)(-1.697,-0.6)(1.59197,-0.84)(-1.44,-1.08)(1.2237,-1.32)(-0.897,-1.56)(0,-1.8)


\psline[linestyle=dotted, linewidth=0.4pt](-2,1.8)(2,1.8)
\psline[linestyle=dotted, linewidth=0.4pt](-2,1.56)(2,1.56)
\psline[linestyle=dotted, linewidth=0.4pt](-2,1.32)(2,1.32)
\psline[linestyle=dotted, linewidth=0.4pt](-2,1.08)(2,1.08)
\psline[linestyle=dotted, linewidth=0.4pt](-2,0.84)(2,0.84)
\psline[linestyle=dotted, linewidth=0.4pt](-2,0.6)(2,0.6)
\psline[linestyle=dotted, linewidth=0.4pt](-2,0.36)(2,0.36)
\psline[linestyle=dotted, linewidth=0.4pt](-2,0.12)(2,0.12)
\psline[linestyle=dotted, linewidth=0.4pt](-2,-1.8)(2,-1.8)
\psline[linestyle=dotted, linewidth=0.4pt](-2,-1.56)(2,-1.56)
\psline[linestyle=dotted, linewidth=0.4pt](-2,-1.32)(2,-1.32)
\psline[linestyle=dotted, linewidth=0.4pt](-2,-1.08)(2,-1.08)
\psline[linestyle=dotted, linewidth=0.4pt](-2,-0.84)(2,-0.84)
\psline[linestyle=dotted, linewidth=0.4pt](-2,-0.6)(2,-0.6)
\psline[linestyle=dotted, linewidth=0.4pt](-2,-0.36)(2,-0.36)
\psline[linestyle=dotted, linewidth=0.4pt](-2,-0.12)(2,-0.12)

\rput(2.2, 1.8){\tiny $-8$}
\rput(2.2, 1.56){\tiny $-7$}
\rput(2.2, 1.32){\tiny $-6$}
\rput(2.2, 1.08){\tiny $-5$}
\rput(2.2, 0.84){\tiny $-4$}
\rput(2.2, 0.6){\tiny $-3$}
\rput(-0.4,0){\tiny $-2$}
\rput(0.4,0){\tiny $2$}
\rput(2.2, 0.12){\tiny $-1$}
\rput(2.2, -0.12){\tiny $1$}
\rput(2.2, -0.6){\tiny $3$}
\rput(2.2, -0.84){\tiny $4$}
\rput(2.2, -1.08){\tiny $5$}
\rput(2.2, -1.32){\tiny $6$}
\rput(2.2, -1.56){\tiny $7$}
\rput(2.2, -1.8){\tiny $8$}
\end{pspicture} &

\begin{pspicture}(-2,-1.92)(2.5,1.92)
\pscircle[linecolor=gray, linewidth=0.2pt](0,0){1.8}

\psline[linecolor=red](-1.2237,-1.32)(-1.59197,0.84)

\psline[linecolor=red](1.44,-1.08)(-0.897,-1.56)(0,-1.8)(1.44,-1.08)

\psline[linecolor=red](1.2237,1.32)(1.59197,-0.84)

\psline[linecolor=red](-1.44,1.08)(0.897,1.56)(0,1.8)(-1.44,1.08)

\psdots(0,1.8)(0.897,1.56)(-1.2237,-1.32)(1.44,-1.08)(-1.59197,0.84)(1.697,0.6)(0,0)(1.796,0.12)(-1.796,-0.12)(-1.697,-0.6)(1.59197,-0.84)(-1.44,1.08)(1.2237,1.32)(-0.897,-1.56)(0,-1.8)


\psline[linestyle=dotted, linewidth=0.4pt](-2,1.8)(2,1.8)
\psline[linestyle=dotted, linewidth=0.4pt](-2,1.56)(2,1.56)
\psline[linestyle=dotted, linewidth=0.4pt](-2,1.32)(2,1.32)
\psline[linestyle=dotted, linewidth=0.4pt](-2,1.08)(2,1.08)
\psline[linestyle=dotted, linewidth=0.4pt](-2,0.84)(2,0.84)
\psline[linestyle=dotted, linewidth=0.4pt](-2,0.6)(2,0.6)
\psline[linestyle=dotted, linewidth=0.4pt](-2,0.36)(2,0.36)
\psline[linestyle=dotted, linewidth=0.4pt](-2,0.12)(2,0.12)
\psline[linestyle=dotted, linewidth=0.4pt](-2,-1.8)(2,-1.8)
\psline[linestyle=dotted, linewidth=0.4pt](-2,-1.56)(2,-1.56)
\psline[linestyle=dotted, linewidth=0.4pt](-2,-1.32)(2,-1.32)
\psline[linestyle=dotted, linewidth=0.4pt](-2,-1.08)(2,-1.08)
\psline[linestyle=dotted, linewidth=0.4pt](-2,-0.84)(2,-0.84)
\psline[linestyle=dotted, linewidth=0.4pt](-2,-0.6)(2,-0.6)
\psline[linestyle=dotted, linewidth=0.4pt](-2,-0.36)(2,-0.36)
\psline[linestyle=dotted, linewidth=0.4pt](-2,-0.12)(2,-0.12)

\rput(2.2, 1.8){\tiny $-8$}
\rput(2.2, 1.56){\tiny $-7$}
\rput(2.2, 1.32){\tiny $-6$}
\rput(2.2, 1.08){\tiny $-5$}
\rput(2.2, 0.84){\tiny $-4$}
\rput(2.2, 0.6){\tiny $-3$}
\rput(-0.4,0){\tiny $-2$}
\rput(0.4,0){\tiny $2$}
\rput(2.2, 0.12){\tiny $-1$}
\rput(2.2, -0.12){\tiny $1$}
\rput(2.2, -0.6){\tiny $3$}
\rput(2.2, -0.84){\tiny $4$}
\rput(2.2, -1.08){\tiny $5$}
\rput(2.2, -1.32){\tiny $6$}
\rput(2.2, -1.56){\tiny $7$}
\rput(2.2, -1.8){\tiny $8$}
\end{pspicture}

\end{tabular}

\caption{Illustration for the second case in the proof of Theorem~\ref{thm_d}. Here we have $s=s_5=(5,6)(-5,-6)$, $P=\{-4,5,6,7,8\}$ is a polygon of $\Read(y)$ (on the left), which is split into two polygons $P_1=\{5,7,8\}$, $P_2=\{-4,6\}$ of $\Read(y)s$ (viewed as $scs$-noncrossing partition, on the right). The total orders $\prec$ on polygons is the same except that $P$ is replaced by $P_1\prec P_2$ (and $-P$ consistently).}  
\label{fig:d_n_1}

\end{figure}

\begin{figure}  

\begin{tabular}{cc}

\psscalebox{1.01}{\begin{pspicture}(-2,-1.92)(2.5,1.8)
\pscircle[linecolor=gray, linewidth=0.2pt](0,0){1.8}

\psline[linecolor=blue](-0.897,-1.56)(0,0.32)
\pscurve[linecolor=blue](-1.44,-1.04)(-0.8,-0.35)(0,0.36)
\pscurve[linecolor=blue](-0.92,-1.53)(-1.15,-1.1)(-1.44,-1.08)

\psline[linecolor=red](0.897,1.56)(0,-0.32)
\pscurve[linecolor=red](1.44,1.04)(0.8,0.35)(0,-0.36)
\pscurve[linecolor=red](0.92,1.53)(1.15,1.1)(1.44,1.08)

\pscurve[linecolor=red](0,1.8)(-1.1,0.85)(-1.697,-0.56)
\pscurve[linecolor=red](0.02,1.76)(-0.65,0.85)(-1.697,-0.6)

\pscurve[linecolor=red](0,-1.8)(1.1,-0.85)(1.697,0.56)
\pscurve[linecolor=red](-0.02,-1.76)(0.65,-0.85)(1.697,0.6)

\psdots(0,1.8)(0.897,1.56)(-1.2237,1.32)(1.44,1.08)(-1.59197,0.84)(1.697,0.6)(1.796,0.12)(-1.796,-0.12)(0,-0.36)(-1.697,-0.6)(1.59197,-0.84)(1.2237,-1.32)(0,-1.8)

\psdots[fillcolor=blue, linecolor=blue](1.2237,-1.32)(-0.897,-1.56)(0,0.36)(-1.44,-1.08)

\rput(1.5,-1.4){\tiny ${Q}$}
\rput(-0.9, -1){\tiny ${P}$}


\psline[linestyle=dotted, linewidth=0.4pt](-2,1.8)(2,1.8)
\psline[linestyle=dotted, linewidth=0.4pt](-2,1.56)(2,1.56)
\psline[linestyle=dotted, linewidth=0.4pt](-2,1.32)(2,1.32)
\psline[linestyle=dotted, linewidth=0.4pt](-2,1.08)(2,1.08)
\psline[linestyle=dotted, linewidth=0.4pt](-2,0.84)(2,0.84)
\psline[linestyle=dotted, linewidth=0.4pt](-2,0.6)(2,0.6)
\psline[linestyle=dotted, linewidth=0.4pt](-2,0.36)(2,0.36)
\psline[linestyle=dotted, linewidth=0.4pt](-2,0.12)(2,0.12)
\psline[linestyle=dotted, linewidth=0.4pt](-2,-1.8)(2,-1.8)
\psline[linestyle=dotted, linewidth=0.4pt](-2,-1.56)(2,-1.56)
\psline[linestyle=dotted, linewidth=0.4pt](-2,-1.32)(2,-1.32)
\psline[linestyle=dotted, linewidth=0.4pt](-2,-1.08)(2,-1.08)
\psline[linestyle=dotted, linewidth=0.4pt](-2,-0.84)(2,-0.84)
\psline[linestyle=dotted, linewidth=0.4pt](-2,-0.6)(2,-0.6)
\psline[linestyle=dotted, linewidth=0.4pt](-2,-0.36)(2,-0.36)
\psline[linestyle=dotted, linewidth=0.4pt](-2,-0.12)(2,-0.12)

\rput(2.2, 1.8){\tiny $-8$}
\rput(2.2, 1.56){\tiny $-7$}
\rput(2.2, 1.32){\tiny $-6$}
\rput(2.2, 1.08){\tiny $-5$}
\rput(2.2, 0.84){\tiny $-4$}
\rput(2.2, 0.6){\tiny $-3$}
\rput(2.2, 0.36){\tiny $-2$}
\rput(2.2, 0.12){\tiny $-1$}
\rput(2.2, -0.12){\tiny $1$}
\rput(2.2, -0.36){\tiny $2$}
\rput(2.2, -0.6){\tiny $3$}
\rput(2.2, -0.84){\tiny $4$}
\rput(2.2, -1.08){\tiny $5$}
\rput(2.2, -1.32){\tiny $6$}
\rput(2.2, -1.56){\tiny $7$}
\rput(2.2, -1.8){\tiny $8$}
\end{pspicture}}

&

\psscalebox{1.01}{\begin{pspicture}(-2,-1.92)(2.5,1.8)
\pscircle[linecolor=gray, linewidth=0.2pt](0,0){1.8}

\psline[linecolor=blue](-0.897,-1.56)(0,0.32)
\pscurve[linecolor=blue](-1.2237,-1.32)(-0.8,-0.35)(0,0.36)
\pscurve[linecolor=blue](-0.92,-1.53)(-1.15,-1.3)(-1.2237,-1.32)

\psline[linecolor=red](0.897,1.56)(0,-0.32)
\pscurve[linecolor=red](1.2237,1.32)(0.8,0.35)(0,-0.36)
\pscurve[linecolor=red](0.92,1.53)(1.15,1.3)(1.2237,1.32)

\pscurve[linecolor=red](0,1.8)(-1.1,0.85)(-1.697,-0.56)
\pscurve[linecolor=red](0.02,1.76)(-0.65,0.85)(-1.697,-0.6)

\pscurve[linecolor=red](0,-1.8)(1.1,-0.85)(1.697,0.56)
\pscurve[linecolor=red](-0.02,-1.76)(0.65,-0.85)(1.697,0.6)

\psdots(0,1.8)(-1.44,1.08)(-1.59197,0.84)(1.697,0.6)(1.796,0.12)(-1.796,-0.12)(-1.697,-0.6)(1.59197,-0.84)(0,-1.8)(0.897,1.56)(0,-0.36)(1.2237,1.32)

\psdots[fillcolor=blue, linecolor=blue](1.44,-1.08)(-0.897,-1.56)(0,0.36)(-1.2237,-1.32)

\rput(1.7,-1.08){\tiny $\overline{Q}$}
\rput(-0.9, -1){\tiny $\overline{P}$}


\psline[linestyle=dotted, linewidth=0.4pt](-2,1.8)(2,1.8)
\psline[linestyle=dotted, linewidth=0.4pt](-2,1.56)(2,1.56)
\psline[linestyle=dotted, linewidth=0.4pt](-2,1.32)(2,1.32)
\psline[linestyle=dotted, linewidth=0.4pt](-2,1.08)(2,1.08)
\psline[linestyle=dotted, linewidth=0.4pt](-2,0.84)(2,0.84)
\psline[linestyle=dotted, linewidth=0.4pt](-2,0.6)(2,0.6)
\psline[linestyle=dotted, linewidth=0.4pt](-2,0.36)(2,0.36)
\psline[linestyle=dotted, linewidth=0.4pt](-2,0.12)(2,0.12)
\psline[linestyle=dotted, linewidth=0.4pt](-2,-1.8)(2,-1.8)
\psline[linestyle=dotted, linewidth=0.4pt](-2,-1.56)(2,-1.56)
\psline[linestyle=dotted, linewidth=0.4pt](-2,-1.32)(2,-1.32)
\psline[linestyle=dotted, linewidth=0.4pt](-2,-1.08)(2,-1.08)
\psline[linestyle=dotted, linewidth=0.4pt](-2,-0.84)(2,-0.84)
\psline[linestyle=dotted, linewidth=0.4pt](-2,-0.6)(2,-0.6)
\psline[linestyle=dotted, linewidth=0.4pt](-2,-0.36)(2,-0.36)
\psline[linestyle=dotted, linewidth=0.4pt](-2,-0.12)(2,-0.12)

\rput(2.2, 1.8){\tiny $-8$}
\rput(2.2, 1.56){\tiny $-7$}
\rput(2.2, 1.32){\tiny $-6$}
\rput(2.2, 1.08){\tiny $-5$}
\rput(2.2, 0.84){\tiny $-4$}
\rput(2.2, 0.6){\tiny $-3$}
\rput(2.2, 0.36){\tiny $-2$}
\rput(2.2, 0.12){\tiny $-1$}
\rput(2.2, -0.12){\tiny $1$}
\rput(2.2, -0.36){\tiny $2$}
\rput(2.2, -0.6){\tiny $3$}
\rput(2.2, -0.84){\tiny $4$}
\rput(2.2, -1.08){\tiny $5$}
\rput(2.2, -1.32){\tiny $6$}
\rput(2.2, -1.56){\tiny $7$}
\rput(2.2, -1.8){\tiny $8$}
\end{pspicture}}

\end{tabular}

\caption{Illustration for the third case in the proof of Theorem~\ref{thm_d}. Here $s=s_5=(5,6)(-5,-6)$. The order on the polygons of $\Read(y)$ (on the left) is given in Example~\ref{ex_dn}. In the notation of the proof we have $P=\{-2,5,7\}$, $Q=\{6\}$, $\overline{P}=\{-2,6,7\}$, $\overline{Q}=\{5\}$. The order $\prec$ on the polygons of $s\Read(y)s$ (represented on the right) is the same as the order on the polygons of $\Read(y)$ except that $P$ is replaced by $\overline{P}$ and $Q$ by $\overline{Q}$ (and $-P$, $-Q$ consistently). } 
\label{fig:d_n_autre}

\end{figure}
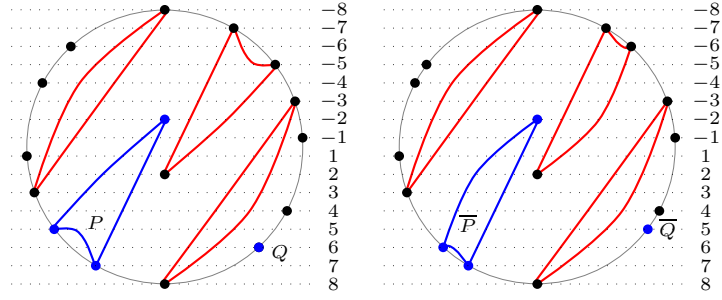	

\begin{figure}  

\begin{tabular}{cc}

\psscalebox{1.01}{\begin{pspicture}(-2,-1.92)(2.5,1.8)
\pscircle[linecolor=gray, linewidth=0.2pt](0,0){1.8}





\psline[linecolor=red](0,1.8)(0.897,1.56)
\psline[linecolor=red](-1.697,-0.6)(0.897,1.56)
\psline[linecolor=red](0,1.8)(-1.697,-0.6)

\psline[linecolor=red](0,-1.8)(-0.897,-1.56)
\psline[linecolor=red](1.697,0.6)(-0.897,-1.56)
\psline[linecolor=red](0,-1.8)(1.697,0.6)

\pscurve[linecolor=blue](1.79,-0.32)(1.45,-0.58)(1.2237,-1.32)

\pscurve[linecolor=red](-1.79,0.32)(-1.45,0.58)(-1.2237,1.32)

\psdots[fillcolor=blue, linecolor=blue](1.2237,-1.32)(0,0.16)(1.79,-0.32)

\psdots(0,1.8)(0.897,1.56)(-1.2237,1.32)(1.44,1.08)(-1.59197,0.84)(1.697,0.6)(-1.79,0.32)(0,-0.16)(-1.697,-0.6)(1.59197,-0.84)(-1.44,-1.08)(-0.897,-1.56)(0,-1.8)

\rput(0.3,0.2){\tiny $P$}
\rput(1.35,-0.4){\tiny $Q$}


\psline[linestyle=dotted, linewidth=0.4pt](-2,1.8)(2,1.8)
\psline[linestyle=dotted, linewidth=0.4pt](-2,1.56)(2,1.56)
\psline[linestyle=dotted, linewidth=0.4pt](-2,1.32)(2,1.32)
\psline[linestyle=dotted, linewidth=0.4pt](-2,1.08)(2,1.08)
\psline[linestyle=dotted, linewidth=0.4pt](-2,0.84)(2,0.84)
\psline[linestyle=dotted, linewidth=0.4pt](-2,0.6)(2,0.6)
\psline[linestyle=dotted, linewidth=0.4pt](-2,0.36)(2,0.36)
\psline[linestyle=dotted, linewidth=0.4pt](-2,0.12)(2,0.12)
\psline[linestyle=dotted, linewidth=0.4pt](-2,-1.8)(2,-1.8)
\psline[linestyle=dotted, linewidth=0.4pt](-2,-1.56)(2,-1.56)
\psline[linestyle=dotted, linewidth=0.4pt](-2,-1.32)(2,-1.32)
\psline[linestyle=dotted, linewidth=0.4pt](-2,-1.08)(2,-1.08)
\psline[linestyle=dotted, linewidth=0.4pt](-2,-0.84)(2,-0.84)
\psline[linestyle=dotted, linewidth=0.4pt](-2,-0.6)(2,-0.6)
\psline[linestyle=dotted, linewidth=0.4pt](-2,-0.36)(2,-0.36)
\psline[linestyle=dotted, linewidth=0.4pt](-2,-0.12)(2,-0.12)

\rput(2.2, 1.8){\tiny $-8$}
\rput(2.2, 1.56){\tiny $-7$}
\rput(2.2, 1.32){\tiny $-6$}
\rput(2.2, 1.08){\tiny $-5$}
\rput(2.2, 0.84){\tiny $-4$}
\rput(2.2, 0.6){\tiny $-3$}
\rput(2.2, 0.36){\tiny $-2$}
\rput(2.2, 0.12){\tiny $-1$}
\rput(2.2, -0.12){\tiny $1$}
\rput(2.2, -0.36){\tiny $2$}
\rput(2.2, -0.6){\tiny $3$}
\rput(2.2, -0.84){\tiny $4$}
\rput(2.2, -1.08){\tiny $5$}
\rput(2.2, -1.32){\tiny $6$}
\rput(2.2, -1.56){\tiny $7$}
\rput(2.2, -1.8){\tiny $8$}
\end{pspicture}}

&

\psscalebox{1.01}{\begin{pspicture}(-2,-1.92)(2.5,1.8)
\pscircle[linecolor=gray, linewidth=0.2pt](0,0){1.8}





\psline[linecolor=red](0,1.8)(0.897,1.56)
\psline[linecolor=red](-1.697,-0.6)(0.897,1.56)
\psline[linecolor=red](0,1.8)(-1.697,-0.6)

\psline[linecolor=red](0,-1.8)(-0.897,-1.56)
\psline[linecolor=red](1.697,0.6)(-0.897,-1.56)
\psline[linecolor=red](0,-1.8)(1.697,0.6)

\pscurve[linecolor=blue](1.796,0.12)(1.45,-0.58)(1.2237,-1.32)

\pscurve[linecolor=red](-1.796,-0.12)(-1.45,0.58)(-1.2237,1.32)

\psdots[linecolor=blue, fillcolor=blue](0,-0.36)(1.796,0.12)(1.2237,-1.32)

\rput(-0.35,-0.36){\tiny $\overline{P}$}
\rput(1.32,-0.4){\tiny $\overline{Q}$}

\psdots(0,1.8)(0.897,1.56)(-1.2237,1.32)(1.44,1.08)(-1.59197,0.84)(1.697,0.6)(-1.796,-0.12)(0,0.36)(-1.697,-0.6)(1.59197,-0.84)(-1.44,-1.08)(-0.897,-1.56)(0,-1.8)


\psline[linestyle=dotted, linewidth=0.4pt](-2,1.8)(2,1.8)
\psline[linestyle=dotted, linewidth=0.4pt](-2,1.56)(2,1.56)
\psline[linestyle=dotted, linewidth=0.4pt](-2,1.32)(2,1.32)
\psline[linestyle=dotted, linewidth=0.4pt](-2,1.08)(2,1.08)
\psline[linestyle=dotted, linewidth=0.4pt](-2,0.84)(2,0.84)
\psline[linestyle=dotted, linewidth=0.4pt](-2,0.6)(2,0.6)
\psline[linestyle=dotted, linewidth=0.4pt](-2,0.36)(2,0.36)
\psline[linestyle=dotted, linewidth=0.4pt](-2,0.12)(2,0.12)
\psline[linestyle=dotted, linewidth=0.4pt](-2,-1.8)(2,-1.8)
\psline[linestyle=dotted, linewidth=0.4pt](-2,-1.56)(2,-1.56)
\psline[linestyle=dotted, linewidth=0.4pt](-2,-1.32)(2,-1.32)
\psline[linestyle=dotted, linewidth=0.4pt](-2,-1.08)(2,-1.08)
\psline[linestyle=dotted, linewidth=0.4pt](-2,-0.84)(2,-0.84)
\psline[linestyle=dotted, linewidth=0.4pt](-2,-0.6)(2,-0.6)
\psline[linestyle=dotted, linewidth=0.4pt](-2,-0.36)(2,-0.36)
\psline[linestyle=dotted, linewidth=0.4pt](-2,-0.12)(2,-0.12)

\rput(2.2, 1.8){\tiny $-8$}
\rput(2.2, 1.56){\tiny $-7$}
\rput(2.2, 1.32){\tiny $-6$}
\rput(2.2, 1.08){\tiny $-5$}
\rput(2.2, 0.84){\tiny $-4$}
\rput(2.2, 0.6){\tiny $-3$}
\rput(2.2, 0.36){\tiny $-2$}
\rput(2.2, 0.12){\tiny $-1$}
\rput(2.2, -0.12){\tiny $1$}
\rput(2.2, -0.36){\tiny $2$}
\rput(2.2, -0.6){\tiny $3$}
\rput(2.2, -0.84){\tiny $4$}
\rput(2.2, -1.08){\tiny $5$}
\rput(2.2, -1.32){\tiny $6$}
\rput(2.2, -1.56){\tiny $7$}
\rput(2.2, -1.8){\tiny $8$}
\end{pspicture}}

\end{tabular}

\caption{Illustration for the third case in the proof of Theorem~\ref{thm_d} in the case where $s=s_0$. The first half of the order on the polygons of $\Read(y)$ (on the left) is given by $\{-5\}\prec\{4\}\prec Q \prec \{-3,7,8\}\prec P$, while the first half of $\prec$ for $s\Read(y)s$ (on the right) is given by $\{-5\}\prec\{4\}\prec \overline{Q} \prec \{-3,7,8\}\prec -\overline{P}$. The order is hence not the same after overlining since in one case in the middle we have $P\prec -P$ while in the second case we get $-\overline{P}\prec\overline{P}$, but to obtain $S_{c}(\Read(y))$ from the order, the two centermost entries in the line notation of $S_{c}(\Read(y))'$, which precisely correspond to the two indices in $-P$ and $P=\{-1\}$, must be exchanged.}
\label{fig:d_n_autre_2}

\end{figure}
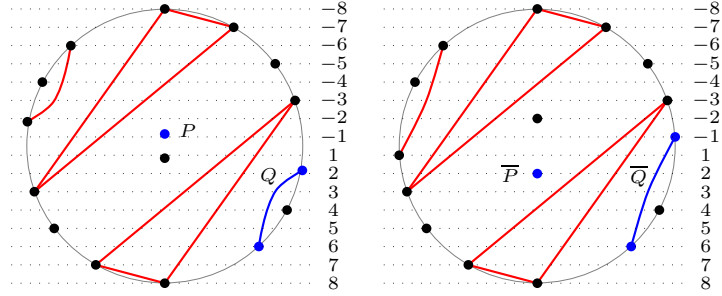

\begin{figure}  

\begin{tabular}{cc}

\psscalebox{1.01}{\begin{pspicture}(-2,-1.92)(2.5,1.8)
\pscircle[linecolor=gray, linewidth=0.2pt](0,0){1.8}





\psline[linecolor=blue](0,1.8)(0.897,1.56)
\psline[linecolor=blue](0.897,1.56)(1.79,-0.32)
\psline[linecolor=blue](1.79,-0.32)(0,1.8)

\psline[linecolor=red](0,-1.8)(-0.897,-1.56)
\psline[linecolor=red](-0.897,-1.56)(-1.79,0.32)
\psline[linecolor=red](-1.79,0.32)(0,-1.8)

\pscurve[linecolor=blue](1.59197,-0.84)(1.45,-0.99)(1.2237,-1.32)
\psline[linecolor=blue](1.59197,-0.84)(0,0.16)
\psline[linecolor=blue](1.2237,-1.32)(0,0.16)

\pscurve[linecolor=red](-1.59197,0.84)(-1.45,0.99)(-1.2237,1.32)
\psline[linecolor=red](-1.59197,0.84)(0,-0.16)
\psline[linecolor=red](-1.2237,1.32)(0,-0.16)


\psdots[fillcolor=blue, linecolor=blue](1.2237,-1.32)(1.79,-0.32)(1.59197,-0.84)(0,0.16)(0,1.8)(0.897,1.56)

\psdots(-1.2237,1.32)(1.44,1.08)(-1.59197,0.84)(1.697,0.6)(-1.79,0.32)(0,-0.16)(-1.697,-0.6)(-1.44,-1.08)(-0.897,-1.56)(0,-1.8)

\rput(0.3,0.2){\tiny $P$}
\rput(0.9,1){\tiny $Q$}


\psline[linestyle=dotted, linewidth=0.4pt](-2,1.8)(2,1.8)
\psline[linestyle=dotted, linewidth=0.4pt](-2,1.56)(2,1.56)
\psline[linestyle=dotted, linewidth=0.4pt](-2,1.32)(2,1.32)
\psline[linestyle=dotted, linewidth=0.4pt](-2,1.08)(2,1.08)
\psline[linestyle=dotted, linewidth=0.4pt](-2,0.84)(2,0.84)
\psline[linestyle=dotted, linewidth=0.4pt](-2,0.6)(2,0.6)
\psline[linestyle=dotted, linewidth=0.4pt](-2,0.36)(2,0.36)
\psline[linestyle=dotted, linewidth=0.4pt](-2,0.12)(2,0.12)
\psline[linestyle=dotted, linewidth=0.4pt](-2,-1.8)(2,-1.8)
\psline[linestyle=dotted, linewidth=0.4pt](-2,-1.56)(2,-1.56)
\psline[linestyle=dotted, linewidth=0.4pt](-2,-1.32)(2,-1.32)
\psline[linestyle=dotted, linewidth=0.4pt](-2,-1.08)(2,-1.08)
\psline[linestyle=dotted, linewidth=0.4pt](-2,-0.84)(2,-0.84)
\psline[linestyle=dotted, linewidth=0.4pt](-2,-0.6)(2,-0.6)
\psline[linestyle=dotted, linewidth=0.4pt](-2,-0.36)(2,-0.36)
\psline[linestyle=dotted, linewidth=0.4pt](-2,-0.12)(2,-0.12)

\rput(2.2, 1.8){\tiny $-8$}
\rput(2.2, 1.56){\tiny $-7$}
\rput(2.2, 1.32){\tiny $-6$}
\rput(2.2, 1.08){\tiny $-5$}
\rput(2.2, 0.84){\tiny $-4$}
\rput(2.2, 0.6){\tiny $-3$}
\rput(2.2, 0.36){\tiny $-2$}
\rput(2.2, 0.12){\tiny $-1$}
\rput(2.2, -0.12){\tiny $1$}
\rput(2.2, -0.36){\tiny $2$}
\rput(2.2, -0.6){\tiny $3$}
\rput(2.2, -0.84){\tiny $4$}
\rput(2.2, -1.08){\tiny $5$}
\rput(2.2, -1.32){\tiny $6$}
\rput(2.2, -1.56){\tiny $7$}
\rput(2.2, -1.8){\tiny $8$}
\end{pspicture}}

&

\psscalebox{1.01}{\begin{pspicture}(-2,-1.92)(2.5,1.8)
\pscircle[linecolor=gray, linewidth=0.2pt](0,0){1.8}





\psline[linecolor=blue](0,1.8)(0.897,1.56)
\psline[linecolor=blue](0.897,1.56)(1.796,0.12)
\psline[linecolor=blue](1.796,0.12)(0,1.8)

\psline[linecolor=red](0,-1.8)(-0.897,-1.56)
\psline[linecolor=red](-0.897,-1.56)(-1.796,-0.12)
\psline[linecolor=red](-1.796,-0.12)(0,-1.8)

\pscurve[linecolor=blue](1.59197,-0.84)(1.45,-0.99)(1.2237,-1.32)
\psline[linecolor=blue](1.59197,-0.84)(0,-0.36)
\psline[linecolor=blue](1.2237,-1.32)(0,-0.36)

\pscurve[linecolor=red](-1.59197,0.84)(-1.45,0.99)(-1.2237,1.32)
\psline[linecolor=red](-1.59197,0.84)(0,0.36)
\psline[linecolor=red](-1.2237,1.32)(0,0.36)

\psdots[fillcolor=blue, linecolor=blue](1.2237,-1.32)(1.796,0.12)(1.59197,-0.84)(0,-0.36)(0,1.8)(0.897,1.56)

\psdots(-1.2237,1.32)(1.44,1.08)(-1.59197,0.84)(1.697,0.6)(-1.796,-0.12)(0,0.36)(-1.697,-0.6)(-1.44,-1.08)(-0.897,-1.56)(0,-1.8)

\rput(0.9,-0.4){\tiny $\overline{P}$}
\rput(0.86,1.2){\tiny $\overline{Q}$}


\psline[linestyle=dotted, linewidth=0.4pt](-2,1.8)(2,1.8)
\psline[linestyle=dotted, linewidth=0.4pt](-2,1.56)(2,1.56)
\psline[linestyle=dotted, linewidth=0.4pt](-2,1.32)(2,1.32)
\psline[linestyle=dotted, linewidth=0.4pt](-2,1.08)(2,1.08)
\psline[linestyle=dotted, linewidth=0.4pt](-2,0.84)(2,0.84)
\psline[linestyle=dotted, linewidth=0.4pt](-2,0.6)(2,0.6)
\psline[linestyle=dotted, linewidth=0.4pt](-2,0.36)(2,0.36)
\psline[linestyle=dotted, linewidth=0.4pt](-2,0.12)(2,0.12)
\psline[linestyle=dotted, linewidth=0.4pt](-2,-1.8)(2,-1.8)
\psline[linestyle=dotted, linewidth=0.4pt](-2,-1.56)(2,-1.56)
\psline[linestyle=dotted, linewidth=0.4pt](-2,-1.32)(2,-1.32)
\psline[linestyle=dotted, linewidth=0.4pt](-2,-1.08)(2,-1.08)
\psline[linestyle=dotted, linewidth=0.4pt](-2,-0.84)(2,-0.84)
\psline[linestyle=dotted, linewidth=0.4pt](-2,-0.6)(2,-0.6)
\psline[linestyle=dotted, linewidth=0.4pt](-2,-0.36)(2,-0.36)
\psline[linestyle=dotted, linewidth=0.4pt](-2,-0.12)(2,-0.12)

\rput(2.2, 1.8){\tiny $-8$}
\rput(2.2, 1.56){\tiny $-7$}
\rput(2.2, 1.32){\tiny $-6$}
\rput(2.2, 1.08){\tiny $-5$}
\rput(2.2, 0.84){\tiny $-4$}
\rput(2.2, 0.6){\tiny $-3$}
\rput(2.2, 0.36){\tiny $-2$}
\rput(2.2, 0.12){\tiny $-1$}
\rput(2.2, -0.12){\tiny $1$}
\rput(2.2, -0.36){\tiny $2$}
\rput(2.2, -0.6){\tiny $3$}
\rput(2.2, -0.84){\tiny $4$}
\rput(2.2, -1.08){\tiny $5$}
\rput(2.2, -1.32){\tiny $6$}
\rput(2.2, -1.56){\tiny $7$}
\rput(2.2, -1.8){\tiny $8$}
\end{pspicture}}

\end{tabular}

\caption{Illustration for the third case in the proof of Theorem~\ref{thm_d} in the case where $s=s_0$. The first half of the order on the polygons of $\Read(y)$ (on the left) is given by $\{-5\}\prec\{-3\}\prec Q \prec -P,$ (note that the pair $\{P,-P\}$ is special with $-P\prec P$) while the first half of $\prec$ for $s\Read(y)s$ (on the right) is given by $\{-5\}\prec\{-3\}\prec \overline{Q} \prec -\overline{P}.$
The order is hence the same after overlining.}

\label{fig:d_n_autre_3}

\end{figure}
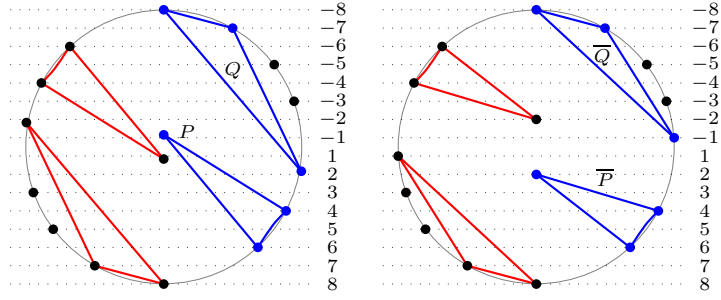	

\end{proof}

The following two technical Lemmatas are type $D_n$ analogues of Lemma~\ref{order_polygons_a}.

\begin{lem}\label{lem:conf}
Let $c$ be a Coxeter element, $i\in\{2, \cdots, n\}$ such that $i\in L_c\cup\{i_1\}$, $i+1\in R_c$. Let $j=c^{-1}(i)$, $k=c^{-1}(i+1)$. Assume that $j\in P_1$, $k\in P_2$, with $P_1\neq P_2$. Then $P_2\prec P_1$.  
\end{lem}

\begin{proof}
If $P_1=-P_1$, then the claim is clear since $P_1$ appears in the middle in the order $\prec$ and the condition $k\in P_2\cap R_c$ forces it to be strictly at the right of the diagonal $(-j)j\in P_1$, hence $P_2 \prec -P_2$. If $P_2=-P_2$, then $P_2$ appears in the middle in $<$ and the condition $j\in P_1\cap L_c$ together with the fact that there are no vertices on the circle between $i+1$ and $k$ forces $-P_1\prec P_1$. So we can assume that $P_i\neq -P_i$ for $i=1,2$.

Note that under these assumptions, $P_2$ is at the right of $P_1$: if the smallest index $\ell$ of $P_1$ is smaller than $k$ then it is clear since the diagonal $\ell k\in P_1$ is strictly at the left of $P_2$, and if not it follows from the fact that there are no indices on the circle between $k$ and $i+1$ as well as between $j$ and $i$. If $P_1\cup-P_1< P_2\cup -P_2$ with $-P_1\prec P_1$ then the result holds by definition of $\prec$. Hence assume that $P_1\cup -P_1 < P_2\cup -P_2$ and $P_1\prec -P_1$. By definition of $\prec$, it means that $P_1$ is at the right of $\{P_1, P_2, -P_2\}$ (note that $\{P_1, -P_1\}$ cannot be special because we would have that $P_1$ contains the miminal index among those of $P_1\cup - P_1$, which would had to be negative because $-j\in -P_1$, and as a consequence $P_1$ would lie strictly at the left of $P_2$). We consider the line segment $L=j\min(P_1)$. If $\ell:=\min(P_1)\in\{i+1, \dots, j-1\}$, then the horizontal line containing the point $\ell$ can intersect no polygon distinct from $P_1$ since $j\in P_1$ and there are no indices between $j$ and $i$ on the circle. Hence $P_1$ is fully disjoint from any polygon with an index smaller than $\ell$, implying that $-P_1$ is fully disjoint from any polygon with an index bigger than $-\ell$, contradicting $P_1\prec - P_1$. Hence we can assume that $\ell:=\min(P_1)\in\{k+1,\dots,i\}\cup\{j\}$. Since $P_1 \prec - P_1$, when removing $P_1\cup -P_1$ in the definition of $<$ we must have a polygon $Q$ which is strictly at the right of $-P_1$ (otherwise $P_1$ and $-P_1$ would be fully disjoint and since $0 \leq j\in P_1$, we would have $-P_1\prec P_1$, a contradiction). This implies that $-Q$ is strictly at the left of $P_1$. But the configuration of the indices is such that either $-Q\subseteq \{ \min(P_1)+1, \dots, i\}\subseteq L_c$, in which case we have $Q\cup -Q< P_1\cup - P_1$ (because $Q$ is in $R_c$ and strictly at the right of $-P_1$), a contradiction, or $\min(Q)> j$, in which case we would again have $Q\cup -Q< P_1\cup - P_1$. Hence we got contradictions in all the cases.



We can therefore assume that $P_2\cup -P_2 < P_1\cup -P_1$. If $P_2\prec -P_2$ then the result holds. Assume that $-P_2\prec P_2$. This means that $-P_2$ is at the right of $-P_1$, hence by symmetry that $P_2$ is at the left of $P_1$, which (since $P_2$ is also at the right of $P_1$) implies that $P_1$ and $P_2$ are fully disjoint. It follows that $-P_2$ cannot be strictly at the right to $P_1$, because in that case $-P_2$ would have an index $\ell > k$, hence the diagonal $(-k)\ell\in - P_2$ would be strictly at the left of $k\in P_2$ contradicting $-P_2\prec P_2$. Hence one can assume that $P_1$ is fully disjoint from both $P_2$ and $- P_2$, and therefore by symmetry we have that $-P_1$ is also disjoint from both $P_2$ and $-P_2$. Arguing similarly as in the first case one sees that it contradicts $P_2\cup-P_2< P_1\cup-P_1$.  
\end{proof}

\begin{lem}\label{lem:confbis}
Let $c$ be a Coxeter element. Assume that we are in one of the following situations
\begin{itemize}
\item We have $i_1=2$, $1\in L_c$, $j=c^{-1}(1)$ with $j\in P_1$, $-2\in P_2$ and $P_2\neq P_1$,
\item We have $i_1=2$, $-1\in L_c$, $j=c^{-1}(-1)$ with $j\in P_1$, $-2\in P_2$ and $P_2\neq P_1$, 
\item We have $i_1=1$, $-2\in L_c$, $k=c^{-1}(2)$ with $k\in P_2$, $- 1\in P_1$ and $P_2\neq P_1$,
\item We have $i_1=1$, $-2\in L_c$, $k=c^{-1}(2)$ with $k\in P_2$, $1\in P_1$ and $P_2\neq P_1$.
\end{itemize}
Then $P_2\prec P_1$.
\end{lem}
\begin{proof}
The proof is similar to the proof of Lemma~\ref{lem:conf}. We only prove it in the first situation. Note that $P_1\neq - P_1$, otherwise we would have $\pm i_1=\pm 2\in P_1$ hence $P_2=P_1$. If $ P_2\cup - P_2 < P_1\cup - P_1$ with $P_2\prec -P_2$ then the result holds. Hence assume that $P_2\cup - P_2 < P_1\cup - P_1$ with $-P_2 \prec P_2$, and find a contradiction. Note that $P_2\neq - P_2$ since a symmetric polygon is bigger than any other pair in the order $<$. By assumption we have that $-P_2$ is at the right of every polygon in $\{ \pm P_1, P_2\}$ except possibly $P_2$ in the case where $\{ \pm P_2 \}$ is a special pair. In particular $- P_2$ is at the right of $-P_1$, hence since $-j\in R_c$ we have $\min (-P_2) > -j$. We show that $-P_2$ and $-P_1$ are fully disjoint. If this is not the case, then $-P_1$ is strictly at the left of $-P_2$, which implies that $P_1$ is strictly at the right of $P_2$. But this is impossible: the configuration of the indices on the circle would then imply that $-1\in P_1$ (note that no index $<-1$ can be in $P_1$ since $P_1$ is strictly at the right of $P_2$ and $P_1$ does not cross $-P_1$; moreover we cannot have $P_1= -P_1$ since we would have $P_1=P_2$), hence that the segment $j (-1)$ is in $P_1$, contradicting the fact that $-P_2$ is at the right of everything unless $-P_2=P_1$. But in that case the pair $\{ \pm P_2\}$ is special with $P_2 \prec -P_2=P_1$, a contradiction. Hence $-P_2$ and $-P_1$ are fully disjoint, and the same holds for $P_2$ and $P_1$. Note that the configuration of the indices is such that we necessarily have $-P_1  \prec P_1$. We then argue as in the proof of Lemma~\ref{order_polygons_a} in type $A_n$: we choose a sequence of polygons $Q_1, \dots, Q_m$ with $-P_2$ strictly at the right of $Q_1$, ..., $Q_m$ strictly at the right of $-P_1$; such a sequence necessarily exists since $-P_2$ and $-P_1$ are fully disjoint but $P_2\cup - P_2 \prec P_1 \cup - P_1$ (and $- P_i \prec P_i$ for $i=1,2$). But this implies that $\{P_2, -P_2\}$ is special with $P_2\prec - P_2$, a contradiction.  

Now assume that $P_1\cup - P_1 < P_2 \cup - P_2$. If $P_2=-P_2$ then we have $-P_1\prec P_1$ hence we can assume that $P_2\neq - P_2$. If $-P_1\prec P_1$ then there is nothing to prove. Hence assume that $P_1 \prec - P_1$ and find a contradiction. It implies that $P_1$ is at the right of everything (except possibly $-P_1$ in the case where the pair is special). If $P_1 = -P_2$ one sees easily that $-P_1 \prec P_1$ (including in the case where the pair is special since $j, 2\in P_1$). Hence we can assume that $P_1\neq \pm P_2$. We show that $P_1$ and $-P_2$ are fully disjoint. If not, then we necessarily have $\min(P_1)=-1$, otherwise $P_1$ crosses $-P_1$ (and recall that $P_1\neq - P_1$). But in that case we have $P_2=\{-2\}$ since $2,j\in P_1$, which implies that $P_2$ is at the right of all the other polygons among $\{ \pm P_1, - P_2\}$ and strictly at the right of $-P_1$, which is the only one with a smaller minimal index, hence $P_2 \cup - P_2 < P_1\cup - P_1$, a contradiction. Hence we can assume that $P_1$ and $-P_2$ are fully disjoint, hence that $-P_1$ and $P_2$ are also fully disjoint. In particular $P_1$ and $-P_1$ are fully disjoint. If $-P_1$ and $- P_2$ are not fully disjoint, then necessarily $-P_1$ lies at the right of $-P_2$. This contradicts $P_1 \prec -P_1$. Hence we can assume that $-P_1$ and $P_2$ are also fully disjoint. We then have that $-P_1$ is fully disjoint from both $P_2$ and $-P_2$, and the configuration of the indices then again implies that $-P_1\prec P_1$.  

\end{proof}

We now prove Lemma~\ref{lem:casparcasbis} in type $D_n$.

\begin{proof}[Proof of Lemma~\ref{lem:casparcasbis} in type $D_n$]
If $s=s_i$ with $i\geq 2$, then the proof works exactly as in type $A_n$. Hence assume that $s=s_1$. If $i_2=2$ then we have $1\in L_c$ (see Remark~\ref{rmq:initial_d}). We have $x^{-1}sx=(x^{-1}(1), x^{-1}(2))$. Arguing again as in the type $A_n$ proof we see that we have to show that $S_c(y)^{-1}(y(j)) > S_c^{-1}(y(-2))$, where $j$ is the index preceding $1$ on the circle in clockwise order (that is, $j=c^{-1}(1)$). Let $P_1$ be the polygon of $y$ containing $j$ (and hence $y(j)$) and $P_2$ be the polygon containing $-2$ (and hence also $y(-2)$). If $P_1=P_2$ then $j(-2)$ is a diagonal of $P_1$ and we necessarily have $y(j) < y(-2)$, hence $S_c(y)^{-1}(y(j)) > S_c^{-1}(y(-2))$ by definition of $S_c$. Hence assume that $P_1\neq P_2$. To show the claim it suffices to see that $P_2\prec P_1$. This is given by the first point of Lemma~\ref{lem:confbis} above. The other cases are done similarly and lead to all the situations treated in Lemma~\ref{lem:confbis}.   

\end{proof}

\subsection{Type $I_2(m)$}

Let $W=\langle s, t\rangle$ be a finite dihedral group of type $I_2(m)$, $m\geq 3$. Then every element except the longest one has a unique reduced expression. Let $c=st$. The $c$-sortable elements are precisely $e, w_0, t$ and those elements among the remaining ones having their reduced expression beginning by $s$. The map $\Read$ is given by $\Read(e)=e$, $\Read(t)=t$, $\Read(w_0)=st=c$ and 
 $$\Read(\underbrace{st\cdots}_{k~\text{factors}})=
\left\lbrace
\begin{array}{ccc}
\underbrace{st\cdots}_{2k-1~\text{factors}}  & \mbox{if} & 1\leq k\leq \frac{m}{2},\\
\underbrace{ts\cdots}_{2(m-k)+1~\text{factors}} & \mbox{if} & m> k> \frac{m}{2}.
\end{array}\right.$$ 

It follows that $S_c(\underbrace{st\cdots s}_{k~\text{factors}})=\underbrace{st\cdots s}_{(k+1)/2~\text{factors}}$, $1\leq k \leq 2m-3$, $S_c(t)=t$, $S_c(e)=e$, $S_c(st)=w_0$. 

\begin{proof}[Proof of Lemma~\ref{lem:casparcasbis} in type $I_2(m)$]
Let $c=st$. The claim holds if $y=e,t,st$. Hence assume that $y=\underbrace{st\cdots}_{k~\text{factors}}$, $1\leq k\leq 2m-3$. We have $x=\underbrace{st\cdots s}_{k+2~\text{factors}}$, and $x^{-1}sx=\underbrace{st\cdots s}_{2k+3~\text{factors}}$. Assume that $1\leq k\leq m$. In that case, the words $y=\underbrace{st\cdots s}_{k~\text{factors}}$ and $S_c(y)=\underbrace{st\cdots s}_{(k+1)/2~\text{factors}}$ are both reduced. The set of inversions of $y$ is given by $\{s, sts, \dots, \underbrace{st\cdots s}_{2k-1~\text{factors}}\}$ and that of $S_c(y)$ is given by $\{s, sts, \dots, \underbrace{st\cdots s}_{k~\text{factors}}\}$. It follows that if $2k+3< 2m$, then $x^{-1}sx$ is neither an inversion of $y$ nor of $S_c(y)$. If $k=m-1$, then $x^{-1}sx=s$ which is an inversion of both $y$ and $S_c(y)$. If $k=m$, then $x^{-1}sx=sts$ and since $m\geq 3$ it follows that $x^{-1}sx$ is an inversion of both $y$ and $S_c(y)$. Now assume that $m\leq k \leq 2m-3$. Setting $k'=2m-k$ we have $3\leq k'\leq m$ and in that case we get $S_c(y=\underbrace{ts\cdots t}_{k'~\text{factors}})=\underbrace{st\cdots s}_{m-(k'-1)/2~\text{factors}}$ and $x^{-1} sx=\underbrace{st\cdots s}_{2(m-k')+3~\text{factors}}$. Note that we have that $2(m-k')+3 < 2m$ since $3\leq k'$. We have that $x^{-1}sx$ is an inversion of $S_c(y)$ if and only if $2(m-k')+3\leq 2 (m- (k'-1)/2)-1=2m-k'$, which happens if and only if $k'\geq 3$, hence always. Now totally order the reflections of $W$ a follows: $s, sts, ststs, \dots$. The inversions of $y$ are the $k'$ last reflections in this order, while $x^{-1} sx$ appears in position $m-k'+2$. In particular since there are $m$ reflections, we have that $x^{-1} sx$ is always an inversion of $y$ in that case. Hence $x^{-1}sx$ is an inversion of both $y$ and $S_c(y)$ in that case, showing the claim.  

\end{proof}

\subsection{Types $E, F, H$}

Lemma~\ref{lem:casparcasbis} was checked by computer in these types, using the CHEVIE package of GAP 3.  

\section{Questions}

We conclude by giving a list of questions related to this work. 

\begin{question}
Is there a uniform proof of Lemma~\ref{lem:casparcasbis} ? It would give a uniform proof of Theorem~\ref{thm:main}.
\end{question}

\begin{question}
Is there a uniform description of the map $S_c:\NC(W,c)\longrightarrow \Sort_c(W)$ ? 
\end{question}

\begin{question}
Can Theorem~\ref{thm:main} be extended to some Artin-Tits groups of non-spherical type admitting a dual (quasi-) Garside structure ? Such structures exist in types $\widetilde{A}_n$ and $\widetilde{C}_n$ (for suitable choices of Coxeter elements, see~\cite{Dig}, \cite{Dig1}) and for universal Coxeter groups~\cite{Bessis_free}. A theory of $c$-sortable elements for infinite Coxeter groups was developed in~\cite{RS}. Note that an extension to these types would imply that simple dual braids have a positive Kazhdan-Lusztig expansion (Corollary~\ref{dual_kl_positive}), which is not known in these types since it is not known that simple dual braids are Mikado braids (Theorem~\ref{simple_mikado}; note that here, one has to use the definition of Mikado braids for general Artin-Tits groups, which was not given in this paper). 

Indeed, there is no known topological or homological characterization of Mikado braids in types $\widetilde{A}_n$, $\widetilde{C}_n$ or for universal groups, which could be used to prove Theorem~\ref{simple_mikado} (similarly as in either~\cite{DG}, \cite{BG} or \cite{LQ}). 
 
\end{question}

\end{document}